\documentclass[hidelinks,onefignum,onetabnum]{siamart251216}

\usepackage{lipsum}
\usepackage{amsfonts}
\usepackage{graphicx}
\usepackage{epstopdf}
\usepackage{algorithmic}
\ifpdf
  \DeclareGraphicsExtensions{.eps,.pdf,.png,.jpg}
\else
  \DeclareGraphicsExtensions{.eps}
\fi

\newsiamremark{remark}{Remark}
\newsiamremark{hypothesis}{Hypothesis}
\crefname{hypothesis}{Hypothesis}{Hypotheses}
\newsiamthm{claim}{Claim}
\newsiamremark{fact}{Fact}
\crefname{fact}{Fact}{Facts}

\headers{FEM for Localized Nonlocal Problems}{Yuyan Chang, Hui Liang and Zhonghua Qiao}
\title{Finite Element Approximation of Nonlocal Problems with Heterogeneous Localization and Local Boundary Conditions}

\author{
Yuyan Chang\thanks{
School of Science, Harbin Institute of Technology (Shenzhen),
Shenzhen 518055, China;
Department of Applied Mathematics,
The Hong Kong Polytechnic University,
Hung Hom, Kowloon, Hong Kong, China
(\email{yuyan.chang@connect.polyu.hk}).
}
\and
Hui Liang\thanks{
Corresponding author.
School of Science, Harbin Institute of Technology (Shenzhen),
Shenzhen 518055, China
(\email{lianghui@hit.edu.cn}).
}
\and
Zhonghua Qiao\thanks{
Department of Applied Mathematics,
The Hong Kong Polytechnic University,
Hung Hom, Kowloon, Hong Kong, China
(\email{zhonghua.qiao@polyu.edu.hk}).
}
}

\usepackage{amsopn}

\ifpdf
\hypersetup{
  pdftitle={Finite Element Approximation of Nonlocal Problems with Heterogeneous Localization and Local Boundary Conditions},
  pdfauthor={Yuyan Chang, Hui Liang and Zhonghua Qiao}
}
\fi

\newtheorem{assumption}[theorem]{Assumption}
\begin{document}
\maketitle
\begin{abstract}
This paper studies the finite element approximation of a one-dimensional nonlocal Poisson problem with heterogeneous localization and homogeneous local Dirichlet boundary conditions. These local boundary conditions induce localization kernels with spatially varying interaction neighborhoods, which lead to substantial numerical challenges for the assembly of the singular nonlocal stiffness matrix. An asymptotically compatible conforming finite element method is developed for the variational formulation, together with an exact geometric decomposition for
the singular stiffness matrix assembly. Under
additional smoothness assumptions on the localization profile, second-order operator consistency is established and error estimates are derived with the corresponding convergence orders. Numerical experiments confirm the theoretical convergence behavior and demonstrate the improved boundary behavior of the heterogeneous localization model.
\end{abstract}
\section{Introduction}
Nonlocal models describe phenomena governed by finite-range interactions. They have been widely used in anomalous transport, peridynamics, and multiscale material response
\cite{bobaru2009convergence,du2012analysis,du2018peridynamic,d2020numerical,silling2000reformulation,silling2010peridynamic}. In contrast to classical differential models, nonlocal models can capture discontinuities and interface effects more naturally through finite-range interactions. A substantial body of work has studied numerical approximations of
volume-constrained nonlocal diffusion and peridynamic models, including finite element and discontinuous Galerkin methods, meshfree and collocation schemes, and quadrature-based discretizations
\cite{aulisa2022efficient,d2020numerical,d2021cookbook,macek2007peridynamics,silling2005meshfree,tao2017nonlocal,tian2013analysis,tian2015nonconforming}. For nonlocal problems on bounded domains, these classical formulations usually prescribe boundary conditions through volumetric constraints on an exterior interaction region \cite{gunzburger2010nonlocal}. This gives a nonlocal counterpart of classical boundary conditions, but it requires data outside the physical domain and may lead to boundary regularity and consistency issues \cite{du2012analysis,tao2017nonlocal}. Various alternatives have therefore been proposed to incorporate classical surface data, including artificial boundary extensions, modified boundary equations, fictitious- or mirror-node constructions, and variable-horizon treatments \cite{behera2022imposition,prudhomme2020treatment,zhao2025enforcing}.

A more systematic model formulation for incorporating classical boundary data is based on heterogeneous localization, where the interaction varies spatially and vanishes at the boundary \cite{d2022review,seleson2013interface,silling2015variable}. This localization strategy was further developed into a mathematical analytical framework, including trace theorems, Hardy-type inequalities, variational principles, Green's identities, and regularity theory \cite{du2022fractional,scott2024nonlocal,scott2025nonlocal,tian2017trace}.
 Tao, Tian, and Du implemented this idea through spatially varying horizons and studied local Dirichlet and Neumann conditions, including auxiliary-function and smooth-horizon treatments for inhomogeneous conditions \cite{tao2019nonlocal}. Recently, asymptotically compatible Galerkin frameworks provide a basis for discretizations that are convergent for both nonlocal models and their local limits \cite{du2026asymptotically}.

Turning the heterogeneous localization framework into a numerical method is not a straightforward extension of volume-constrained discretizations. Classical volume-constrained models usually have translation-invariant kernels and fixed interaction geometries, whereas the localized kernel considered here has base-point interaction support and shrinks near the boundary. In a finite element discretization, this variable interaction geometry enters directly into the bilinear form and the stiffness matrix. Each entry of the stiffness matrix is a singular double integral over an interaction region determined by both the basis-function supports and the singular kernel. This structure prevents a direct use of standard quadrature and motivates the geometric semi-analytical assembly strategy developed in this work. The difficulties become greater in two dimensions because of the more complicated interaction geometry. We therefore first carry out a complete implementation and analysis in one space dimension. To the best of our knowledge, this is the first detailed numerical realization of the heterogeneous localization framework developed in \cite{scott2024nonlocal,scott2025nonlocal}, combining an
exact geometric decomposition with semi-analytical stiffness-matrix
assembly, quantitative operator-consistency analysis, and finite
element error estimates.

In this paper, we study a one-dimensional linear nonlocal Poisson
problem with a heterogeneous horizon and homogeneous classical
Dirichlet boundary conditions. We use the variational and trace
framework developed for nonlocal problems with local boundary
conditions and construct a conforming piecewise-linear finite element approximation. The main contributions of this paper are as follows:
\begin{itemize}
    \item We develop a conforming finite element scheme for the heterogeneously localized nonlocal Poisson problem with homogeneous local Dirichlet boundary conditions.

    \item We derive a geometric and semi‑analytical procedure for assembling the singular nonlocal stiffness matrix associated with spatially varying interaction neighborhoods.

    \item We prove second-order consistency of the nonlocal operator with the local limit, derive the corresponding variational consistency estimate, and obtain error estimates.

    \item We present numerical experiments that confirm the expected convergence behaviors, illustrate the improved boundary behavior enabled by heterogeneous localization.
\end{itemize}

The remainder of the paper is organized as follows. Section~2 introduces the heterogeneous localization setting, the associated function spaces, and the variational formulation. Section~3 presents the finite element discretization and the geometric matrix‑assembly procedure. 
Section~4 establishes consistency and finite element error estimates. Section~5 reports the numerical experiments. The paper concludes with Section~6.

\section{Model setting and variational formulation}
\label{sec:model-setting-and-variational-formulation}

\subsection{The one-dimensional localized nonlocal Poisson problem}
\label{subsec:nonlocal-poisson-problem}
Throughout the paper, we consider the one-dimensional domain $\Omega=(0,1)$, $f\in L^2(\Omega)$, $\beta\in(0,3)$, and $(\cdot,\cdot)$ denotes the inner product in $L^2(\Omega)$.

Let $\lambda:\overline{\Omega}\to[0,\infty)$ be a generalized distance
function and let $q:[0,\infty)\to[0,\infty)$ be a localization profile. We set
\begin{equation}
\label{eq:eta-definition}
\eta(x):=q(\lambda(x)),
\qquad
\eta_\delta(x):=\delta\eta(x)
=\delta q(\lambda(x)).
\end{equation}
\begin{assumption}
\label{ass:localization}
The functions $q$, $\lambda$ and the horizon parameter $\delta$ satisfy the following conditions.
\begin{enumerate}
\item
The function $q\in C^3([0,\infty))$ satisfies
\[
q(0)=q'(0)=0,
\qquad
q''(0)>0,
\]
and
\[
0<q(r)\le r,\qquad r>0,
\]
\[
0\le q'(r)\le1,
\qquad
0\le q''(r)\le1,
\qquad r\ge0.
\]
Moreover, $q''\in C_b^1([0,\infty))$, and there exists
$C_q\ge1$ such that
\[
q(2r)\le C_q q(r),
\qquad r>0.
\]
\item
The function $\lambda\in C^0(\overline{\Omega})\cap C^2(\Omega)$ satisfies the following conditions. There exist constants
$\kappa_j>0$, $j=0,1,2$, with $\kappa_0\ge1$, such that
\[
\kappa_0^{-1}d(x)
\le \lambda(x)
\le \kappa_0 d(x),
\qquad x\in\overline{\Omega},
\]
\[
|\lambda(x)-\lambda(y)|
\le \kappa_1|x-y|,
\qquad x,y\in\Omega,
\]
and
\[
|\lambda^{(j)}(x)|
\le
\kappa_j d(x)^{1-j},
\]
where $d(x):=\operatorname{dist}(x,\partial\Omega)$.
Moreover, the boundary-normal compatibility condition
\[
\lambda'(0^+)=1,
\qquad
\lambda'(1^-)=-1
\]
is imposed.
\item
The horizon parameter satisfies
\[
0<\delta<\min\{\delta_0,\bar{\delta}_0\},
\]
where
\[
\delta_0
=
\left(
3\max\left\{
1,\kappa_1,C_q\kappa_0^{\log_2 C_q}
\right\}
\right)^{-1},
\]
and $\bar{\delta}_0$ is the smallest positive root of \(
\frac{\delta(1+\kappa_1\delta)}
     {(1-\kappa_1\delta)^2}
=
\frac13.
\)
\end{enumerate}
\end{assumption}

Since $q'(0)=0$ and $q''\in L^\infty([0,\infty))$,
\[
|q'(\lambda(x))|
\le \|q''\|_{L^\infty}\lambda(x)
\le Cd(x).
\]
Hence, by the chain rule and the estimates on $\lambda$,
\[
|\eta'(x)|\le Cd(x),
\qquad
|\eta''(x)|\le C.
\]
Consequently,
\[
\eta=q\circ\lambda\in W^{2,\infty}(\Omega).
\]

\subsection{Energy space and weak problem}
\label{subsec:variational-formulation}

We mainly consider the variational formulation of the nonlocal Dirichlet problem. Let
\begin{equation}
    C_\beta:=\frac{3-\beta}{2},
    \qquad
    w_\delta(x):=
    \frac{C_\beta}{\eta_\delta(x)^{3-\beta}}.
    \label{eq:weight}
\end{equation}
This normalization is chosen so that $C_\beta\int_{-1}^{1}|t|^{2-\beta}\,dt=1$.
Define
\begin{equation}
\label{eq:energy-space}
\mathfrak W^{\beta,2}[\delta;q](\Omega)
:=
\left\{
u\in L^2(\Omega):\
[u]_{\mathfrak W^{\beta,2}[\delta;q](\Omega)}<\infty
\right\},
\end{equation}
where the seminorm is
\begin{equation}
\label{eq:seminorm}
[u]^2_{\mathfrak W^{\beta,2}[\delta;q](\Omega)}
:=
\int_0^1\int_0^1
w_\delta(x)\mathbf 1_{\{|y-x|<\eta_\delta(x)\}}
\frac{|u(x)-u(y)|^2}{|x-y|^\beta}
\,dy\,dx .
\end{equation}
Equipped with the norm
\begin{equation}
\label{eq:nonlocal_energynorm}
\|u\|_{\mathfrak W^{\beta,2}[\delta;q](\Omega)}
:=
(\|u\|^2_{L^2(\Omega)}
+
[u]^2_{\mathfrak W^{\beta,2}[\delta;q](\Omega)})^{\frac{1}{2}},
\end{equation}
the space \(\mathfrak W^{\beta,2}[\delta;q](\Omega)\) is a Hilbert space \cite{scott2024nonlocal,scott2025nonlocal}. We use the standard notation \(H^m(\Omega)\) and \(H_0^1(\Omega)\) for classical Sobolev spaces. 

By the trace theorem of \cite[Theorem~1.3]{scott2024nonlocal}, the trace map \(T\) extends continuously to \(\mathfrak W^{\beta,2}[\delta;q](\Omega)\). We define
\begin{equation}
\mathfrak W_0^{\beta,2}[\delta;q](\Omega)
:=
\left\{
u\in\mathfrak W^{\beta,2}[\delta;q](\Omega):
Tu=0\ \text{on }\partial\Omega
\right\}.
\label{eq:zero-trace-space}
\end{equation}
By \cite[Theorem~5.6]{scott2024nonlocal}, this space coincides with the closure of \(C_c^1(\Omega)\) in \(\mathfrak W^{\beta,2}[\delta;q](\Omega)\). The bilinear form is
\begin{equation}
\label{eq:weakproblem_poi}
\mathcal B_\delta(u,v)
:=
\int_0^1\int_0^1
w_\delta(x)\mathbf 1_{\{|y-x|<\eta_\delta(x)\}}
\frac{(u(x)-u(y))(v(x)-v(y))}{|x-y|^\beta}
\,dy\,dx .
\end{equation}
For $0<\beta<3$, by the Cauchy–Schwarz inequality with respect to the nonlocal energy measure, $\mathcal B_\delta$ is well-defined and continuous on
$\mathfrak W^{\beta,2}[\delta;q](\Omega)$ by construction. Moreover, \(   \mathcal  B_\delta(v,v)=    [v]_{\mathfrak W^{\beta,2}[\delta;q](\Omega)}^2.
\)

The nonlocal homogeneous Dirichlet problem is: find $u_\delta\in \mathfrak W_0^{\beta,2}[\delta;q](\Omega)$ such that
\begin{equation}
   \mathcal B_\delta(u_\delta,v)
    =
    (f,v)
    \qquad
    \text{for every }v\in \mathfrak W_0^{\beta,2}[\delta;q](\Omega).
    \label{eq:nonlocal-weak-problem}
\end{equation}
The nonlocal Poincar\'e inequality of \cite[Theorem~5.7]{scott2024nonlocal} is uniform for $0<\delta<\delta_0$. Hence $\mathcal B_\delta$ is coercive on $\mathfrak W_0^{\beta,2}[\delta;q](\Omega)$, and the Lax--Milgram theorem gives a unique solution of \eqref{eq:nonlocal-weak-problem}.

\subsection{Realized strong operator}
\label{subsec:operator-realization}

The formal operator associated with \(\mathcal B_\delta\) is
\begin{equation}
\label{eq:formal-operator}
\widehat{\mathcal L}_\delta u(x)
:=
\int_0^1
\left(
w_\delta(x)\mathbf 1_{\{|x-y|<\eta_\delta(x)\}}
+
w_\delta(y)\mathbf 1_{\{|x-y|<\eta_\delta(y)\}}
\right)
\frac{u(x)-u(y)}{|x-y|^\beta}
\,dy .
\end{equation}
For \(\varepsilon>0\), let
\begin{equation}
\label{eq:truncated-operator}
\begin{aligned}
\mathcal L_\delta^\varepsilon u(x)
:={}&
\mathbf 1_{\{\eta(x)>\varepsilon\}}
\int_0^1
\Big[
w_\delta(x)
\mathbf 1_{\{
\varepsilon\eta_\delta(x)<|x-y|<\eta_\delta(x)
\}}
\\
&\qquad\qquad
+
w_\delta(y)
\mathbf 1_{\{
\varepsilon\eta_\delta(y)<|x-y|<\eta_\delta(y)
\}}
\Big]
\frac{u(x)-u(y)}{|x-y|^\beta}
\,dy .
\end{aligned}
\end{equation}
 We therefore define \(\mathcal L_\delta\) through the hard-cutoff limit
of the truncated operators. The existence of this realization and its
relation to the bilinear form are established in the following proposition.

\begin{proposition}[Hard-cutoff realization]
\label{prop:hard-cutoff-realization}
Under Assumption~\ref{ass:localization}, for every \(u\in C^2(\overline\Omega)\) and each fixed \(\delta\), the limit
\begin{equation}
\label{eq:limit-operator}
\mathcal L_\delta u
:=
\lim_{\varepsilon\rightarrow0}
\mathcal L_\delta^\varepsilon u
\qquad\text{in }L^2(\Omega)
\end{equation}
exists. Moreover,
\begin{equation}
\label{eq:nonlocal-green}
\mathcal B_\delta(u,v)
=
(\mathcal L_\delta u,v)
\qquad
\text{for every }
v\in\mathfrak W_0^{\beta,2}[\delta;q](\Omega).
\end{equation}
\end{proposition}

\begin{proof}
For \(0<\tau<1/2\), replace the endpoint cutoff \(\mathbf 1_{\{|t|<1\}}\) by
\[
\rho_\tau(t):=\mathbf 1_{\{|t|<1-\tau\}},
\]
and denote the resulting form and operators by \(\mathcal B_{\delta,\tau}\), \(\mathcal L_{\delta,\tau}^{\varepsilon}\), and \(\mathcal L_{\delta,\tau}\). Since
\[
\operatorname{supp}\rho_\tau
\subset[-1+\tau,1-\tau],
\]
the support is compactly contained in \((-1,1)\). Hence \cite[Theorem~1.1 and Corollary~5.3]{scott2025nonlocal} applies for every fixed \(\tau\).

The bounds in \cite[Lemmas~5.1--5.2]{scott2025nonlocal} are uniform for \(0<\tau<1/2\): the kernels are bounded by the endpoint cutoff and coincide on the common inner core \(\{|t|<1/2\}\). Consequently,
\[
\lim_{\varepsilon\rightarrow0}
\sup_{0<\tau<1/2}
\left\|
\mathcal L_{\delta,\tau}^{\varepsilon}u
-
\mathcal L_{\delta,\tau}u
\right\|_{L^2(\Omega)}
=0.
\]
For fixed \(\varepsilon>0\), the diagonal singularity is absent, and
dominated convergence gives
\[
\mathcal L_{\delta,\tau}^{\varepsilon}u
\longrightarrow
\mathcal L_\delta^\varepsilon u
\qquad\text{in }L^2(\Omega)
\quad\text{as }\tau\rightarrow0.
\]
Interchanging the two limits by these estimates proves the existence of \eqref{eq:limit-operator}.

For fixed \(\tau\), the Green identity of \cite[Theorem~1.1]{scott2025nonlocal} and the condition \(Tv=0\) give
\[
\mathcal B_{\delta,\tau}(u,v)
=
(\mathcal L_{\delta,\tau}u,v).
\]
Letting \(\tau\rightarrow0\), using dominated convergence for the form and the preceding \(L^2\)-convergence for the operator, proves
\eqref{eq:nonlocal-green}.
\end{proof}

\begin{remark}
For \(0<\beta<2\), the realized operator agrees almost everywhere
with the absolutely convergent expression
\eqref{eq:formal-operator}; this follows from the argument of
\cite[Corollary~5.4]{scott2025nonlocal}, together with the
hard-cutoff approximation above. For \(2\le\beta<3\),
\eqref{eq:formal-operator} is only a formal representation, and
\(\mathcal L_\delta\) always denotes the realization
\eqref{eq:limit-operator}.
\end{remark}

The variational formulation \eqref{eq:nonlocal-weak-problem} is the primary definition of the nonlocal Dirichlet problem and requires no additional regularity.
If $u\in C^2(\overline\Omega)\cap
\mathfrak W_0^{\beta,2}[\delta;q]$, the nonlocal Green identity and the density of \(C_c^\infty(\Omega)\) in \(L^2(\Omega)\) show that the variational problem is equivalent to \(\mathcal L_\delta u=f\) almost everywhere, with \(Tu=0\) encoded by membership in the zero-trace space. This conditional identification is not used in the variational existence theory.

\section{Finite element approximation and stiffness matrix assembly}
\label{sec:numerical-method} 
In this section, we introduce the conforming finite element discretization and the assembly procedure for the resulting nonlocal stiffness matrix. 
\subsection{Conforming finite element discretization}
\label{subsec:fem-discretization}
Let \(N\ge 1\) be an integer, set \(h:=1/N\), and define the grid points \(x_i:=ih\), \(i=0,\ldots,N\). Let
\[
\mathcal K_h:=\{\tau_i=(x_{i-1},x_i]: i=1,\ldots,N\}
\]
be the associated uniform partition of \(\Omega=(0,1)\). We define the conforming piecewise linear finite element space
\[
V_h
:=
\left\{
v_h\in C^0(\overline{\Omega}):
v_h|_{\tau_i}\in \mathbb P_1(\tau_i)\ \text{for all } \tau_i\in\mathcal K_h,
\quad
v_h(0)=v_h(1)=0
\right\}.
\]
Here \(\mathbb P_1(\tau_i)\) denotes the space of polynomials on \(\tau_i\) of degree 1. It is clear that,
\[
V_h\subset H_0^1(\Omega)\cap \mathfrak W_0^{\beta,2}[\delta;q](\Omega).
\]
 The conforming finite element approximation is to find \(u_{\delta,h}\in V_h\) such that
\begin{equation}
\label{eq:discrete-weak-problem}
\mathcal B_\delta(u_{\delta,h},v_h)=(f,v_h)
\qquad
\forall v_h\in V_h .
\end{equation}
We replace  $\mathfrak W_0^{\beta,2}[\delta;q](\Omega)$ with this conforming finite-dimensional subspace \(V_h\).  Let \(\{\phi_i\}_{i=1}^{N_h}\) be the nodal basis of \(V_h\), 
\[
\phi_i(x)=
\begin{cases}
\dfrac{x-x_{i-1}}{h}, & x\in \tau_i,\\[6pt]
\dfrac{x_{i+1}-x}{h}, & x\in \tau_{i+1},\\[6pt]
0, & \text{otherwise},
\end{cases}
\] and define $\phi_{i,1}=\phi_i|_{\tau_i}$ and $\phi_{i,2}=\phi_i|_{\tau_{i+1}}$.
The stiffness matrix $A=(A_{ij})_{i,j=1}^{N_h}\in \mathbb{R}^{N_h\times N_h}$ is defined by
\begin{equation}\label{eq:stiffness-matrix-compact}
A_{ij}
:=
\int_0^1\int_{-\eta_\delta(x)}^{\eta_\delta(x)}
w_\delta(x)\frac{(\phi_i(x)-\phi_i(x+s))(\phi_j(x)-\phi_j(x+s))}{|s|^\beta}
\,ds\,dx.
\end{equation}
The rest of this section develops an exact geometric decomposition for assembling the stiffness matrix \(A\).
\subsection{Geometric decomposition of the stiffness matrix}
\label{subsec:matrix-assembly}
Due to the spatially varying horizon and singular kernel, standard quadrature-based assembly strategies for nonlocal finite element discretizations~\cite{aulisa2022efficient,d2020numerical,d2021cookbook} cannot be directly applied. We assemble the matrix entries according to their diagonal, near-diagonal, and off-diagonal structures, evaluating the inner singular integral analytically and the outer integral by numerical quadrature.

Every interior matrix entry is assembled as a sum of semi-analytical contributions:
\begin{equation}\label{eq:general_Aij}
A_{ij}
=
\sum_{r=1}^{M_{ij}}
\int_{\Omega_{ij}^{(r)}}
w_{\delta}(x)\,
\mathcal I\!\left(
a_{ij}^{(r)}(x), b_{ij}^{(r)}(x);
\mathbf c_{ij,+}^{(r)}(x), \mathbf c_{ij,-}^{(r)}(x)
\right)\,dx .
\end{equation}
Here $\Omega_{ij}^{(r)}$ denotes a base-point subregion, and $a_{ij}^{(r)}(x)$ and $b_{ij}^{(r)}(x)$ are the signed endpoints of the corresponding inner integration interval in the variable $s=y-x$. The vectors $\mathbf c_{ij,+}^{(r)}(x)$ and $\mathbf c_{ij,-}^{(r)}(x)$ are the polynomial coefficient vectors on the right and left sides of the singular point $s=0$. 

Here, we introduce important definitions related to \eqref{eq:general_Aij}. For $s>0$ and $\mathbf c=(c_0,c_1,c_2)$, set $\mathcal P_1(s;\mathbf c) : =c_0\mathcal F(-\beta,s)+c_1\mathcal F(1-\beta,s)+c_2\mathcal F(2-\beta,s)$,
where
\[
\mathcal F(\alpha,s)
:=
\begin{cases}
\dfrac{s^{\alpha+1}}{\alpha+1}, & \alpha\ne -1,\\[6pt]
\ln s, & \alpha=-1,
\end{cases}
\qquad s>0.
\]
Define
\[
\Delta\mathcal P(a,b;\mathbf c)
:=
\begin{cases}
\mathcal P_1(b;\mathbf c)-\mathcal P_1(a;\mathbf c), & a<b,\\[4pt]
0, & a\ge b.
\end{cases}
\]
Then, for two coefficient vectors $\mathbf c_+$ and $\mathbf c_-$, the unified interaction operator is defined by
\begin{equation}\label{eq:unified_I}
\mathcal I(a,b;\mathbf c_+,\mathbf c_-)
:=
\Delta\mathcal P(a_+,b_+;\mathbf c_+)
+
\Delta\mathcal P((-b)_+,(-a)_+;\mathbf c_-),
\end{equation}
where $r_+:=\max\{r,0\}$.
\paragraph{Diagonal entries}
We first consider the diagonal entry
\begin{equation}
\label{eq:Aii-Qi}
A_{ii}
=
\int_{\Omega} w_\delta(x)\mathcal Q_i(x)\,dx,
\qquad
\mathcal Q_i(x)
:=
\int_{-\eta_\delta(x)}^{\eta_\delta(x)}
\frac{(\phi_i(x+s)-\phi_i(x))^2}{|s|^\beta}\,ds.
\end{equation}
For a fixed base point $x$, the inner integral with respect to $s$ is split according to the support intervals containing $x+s$. The numerator is a quadratic polynomial and the singular integral can be evaluated by \eqref{eq:unified_I}.

We split the base-point domain according to whether \(x\) lies in the support of \(\phi_i\):
\[
\Omega_{{\rm in},i}:=[x_{i-1},x_{i+1}],
\qquad
\Omega_{{\rm out},i}:=(0,x_{i-1})\cup(x_{i+1},1).
\]
For brevity, write \(\eta_x:=\eta_\delta(x)\). The clipped shifted intervals are defined:
\begin{align*}
a_{i,L} &:= \max\{x_{i-1}-x,\,-\eta_\delta(x)\},
&
b_{i,L} &:= \min\{x_i-x,\,\eta_\delta(x)\},\\
a_{i,R} &:= \max\{x_i-x,\,-\eta_\delta(x)\},
&
b_{i,R} &:= \min\{x_{i+1}-x,\,\eta_\delta(x)\},\\
a_{i,O1} &:= \max\{-x,\,-\eta_\delta(x)\},
&
b_{i,O1} &:= \min\{x_{i-1}-x,\,0\},\\
a_{i,O2} &:= \max\{x_{i+1}-x,\,0\},
&
b_{i,O2} &:= \min\{1-x,\,\eta_\delta(x)\}.
\end{align*}
We denote these intervals by
\[
(a_{i,L},b_{i,L}),\quad
(a_{i,R},b_{i,R}),\quad
(a_{i,O_1},b_{i,O_1}),\quad
(a_{i,O_2},b_{i,O_2}).
\]

For \(x\in\Omega_{{\rm in},i}\), the coefficient vectors associated with the \(i\)-th basis function are
\begin{align*}
\mathbf{c}_{i,L}^{\pm}
&:=
\left(
\bigl(\phi_{i,1}(x)-\phi_i(x)\bigr)^2,\,
\pm \frac{2\bigl(\phi_{i,1}(x)-\phi_i(x)\bigr)}{h},\,
\frac{1}{h^2}
\right),\\
\mathbf{c}_{i,R}^{\pm}
&:=
\left(
\bigl(\phi_i(x)-\phi_{i,2}(x)\bigr)^2,\,
\pm \frac{2\bigl(\phi_i(x)-\phi_{i,2}(x)\bigr)}{h},\,
\frac{1}{h^2}
\right),\\
\mathbf{c}_{i,O}
&:=
\left(
\phi_i(x)^2,\,
0,\,
0
\right).
\end{align*}
Thus
\[
\begin{aligned}
I_{{\rm in},i}(x)
:={}&
\mathcal I(a_{i,L},b_{i,L};
\mathbf c_{i,L}^{+},\mathbf c_{i,L}^{-})
+
\mathcal I(a_{i,R},b_{i,R};
\mathbf c_{i,R}^{+},\mathbf c_{i,R}^{-})
\\
&+
\mathcal I(a_{i,O_1},b_{i,O_1};
\mathbf c_{i,O},\mathbf c_{i,O})
+
\mathcal I(a_{i,O_2},b_{i,O_2};
\mathbf c_{i,O},\mathbf c_{i,O}).
\end{aligned}
\]
For \(x\in\Omega_{{\rm out},i}\), one has \(\phi_i(x)=0\), and only the translated support of \(\phi_i(x+s)\) contributes. The coefficient vectors are

\[
\mathbf d_{i,1}^{+}
=
\left(
\phi_{i,1}(x)^2,\,
\frac{2\phi_{i,1}(x)}{h},\,
\frac1{h^2}
\right),
\qquad
\mathbf d_{i,1}^{-}
=
\left(
\phi_{i,1}(x)^2,\,
-\frac{2\phi_{i,1}(x)}{h},\,
\frac1{h^2}
\right),
\]
\[
\mathbf d_{i,2}^{+}
=
\left(
\phi_{i,2}(x)^2,\,
-\frac{2\phi_{i,2}(x)}{h},\,
\frac1{h^2}
\right),
\quad\;
\mathbf d_{i,2}^{-}
=
\left(
\phi_{i,2}(x)^2,\,
\frac{2\phi_{i,2}(x)}{h},\,
\frac1{h^2}
\right).
\]
Accordingly,
\[
I_{{\rm out},i}(x)
:=
\mathcal I(a_{i,L},b_{i,L};
\mathbf d_{i,1}^{+},\mathbf d_{i,1}^{-})
+
\mathcal I(a_{i,R},b_{i,R};
\mathbf d_{i,2}^{+},\mathbf d_{i,2}^{-}).
\]
Combining the two base-point regions gives
\begin{equation}
\label{eq:Aii-final}
A_{ii}
=
\int_{\Omega_{{\rm in},i}}
w_\delta(x)I_{{\rm in},i}(x)\,dx
+
\int_{\Omega_{{\rm out},i}}
w_\delta(x) I_{{\rm out},i}(x)\,dx .
\end{equation}

\paragraph{Near off-diagonal entries}
We next consider the adjacent entry \(A_{i,i-1}\). In contrast to the diagonal case, the supports of \(\phi_i\) and \(\phi_{i-1}\) overlap on the element \([x_{i-1},x_i]\). Hence the base-point domain is split into regions where the two basis functions are disconnected, connected through one shifted element, or overlap directly. Here, we define 
\begin{equation}\label{eq:mh}
m_h:=\left\lceil \max_{x\in\Omega}\frac{\eta_\delta(x)}{h}\right\rceil+1,
\end{equation}
and use the base-point intervals
\[
B_1:=(x_{i-1-m_h},x_{i-2}],\qquad
B_2:=(x_{i+1},x_{i+m_h}],
\]
and
\[
C_1:=\tau_{i-1},\qquad
C_2:=\tau_i,\qquad
C_3:=\tau_{i+1}.
\]
All intervals are understood to be intersected with \(\Omega\); empty intervals contribute zero.

For the disconnected regions \(B_1\cup B_2\), only the shifted overlap of the
two basis functions contributes. The clipped interval is $(a_i,b_i)$, 
\[
a_i:=\max\{-\eta_\delta(x),x_{i-1}-x\},
\qquad
b_i:=\min\{\eta_\delta(x),x_i-x\},
\]
with coefficient vector
\[
\mathbf c_i
:=
\left(
\phi_{i-1,2}(x)\phi_{i,1}(x),\,
\frac{\phi_{i-1,2}(x)-\phi_{i,1}(x)}{h},\,
-\frac1{h^2}
\right).
\]
Thus
\[
I_B(x):=\mathcal I(a_i,b_i;\mathbf c_i,\mathbf c_i).
\]
For the connected regions \(C_1\) and \(C_3\), the shifted intervals are \((a_{i,c_1},b_{i,c_1})\), \((a_{i,c_2},b_{i,c_2})\) and
\((a_{i,c_3},b_{i,c_3})\), here
\begin{align*}
a_{i,c1} &:= \max\{-\eta_\delta(x),\,x_{i-2}-x\},
&
b_{i,c1} &:= \min\{\eta_\delta(x),\,x_{i-1}-x\},\\
a_{i,c2} &:= \max\{-\eta_\delta(x),\,x_{i-1}-x\},
&
b_{i,c2} &:= \min\{\eta_\delta(x),\,x_i-x\},\\
a_{i,c3} &:= \max\{-\eta_\delta(x),\,x_i-x\},
&
b_{i,c3} &:= \min\{\eta_\delta(x),\,x_{i+1}-x\}.
\end{align*}
The coefficient vectors are
\begin{align*}
\mathbf{c}_{i,c1}
&=
\left(
\bigl(\phi_{i-1,2}(x)-\phi_{i-1,1}(x)\bigr)\phi_{i,1}(x),\,
\frac{\phi_{i-1,2}(x)-\phi_{i-1,1}(x)-\phi_{i,1}(x)}{h},\,
-\frac{1}{h^2}
\right),\\
\mathbf{c}_{i,c2}
&=
\left(
-\phi_{i,2}(x)\phi_{i-1,1}(x),\,
\frac{\phi_{i-1,1}(x)}{h},\,
0
\right),\\
\mathbf{c}_{i,c3}
&=
\left(
\phi_{i,2}(x)\phi_{i-1,1}(x),\,
\frac{\phi_{i,2}(x)}{h},\,
0
\right),\\
\mathbf{c}_{i,c4}
&=
\left(
\bigl(\phi_{i,1}(x)-\phi_{i,2}(x)\bigr)\phi_{i-1,2}(x),\,
\frac{\phi_{i-1,2}(x)-\phi_{i-1,1}(x)+\phi_{i,2}(x)}{h},\,
-\frac{1}{h^2}
\right).
\end{align*}
Therefore,
\[
I_{C_1}(x)
=
\mathcal I(a_{i,c_2},b_{i,c_2};\mathbf c_{i,c_1},\mathbf 0)
+
\mathcal I(a_{i,c_3},b_{i,c_3};\mathbf c_{i,c_2},\mathbf 0),
\]
\[
I_{C_3}(x)
=
\mathcal I(a_{i,c_1},b_{i,c_1};\mathbf 0,\mathbf c_{i,c_3})
+
\mathcal I(a_{i,c_2},b_{i,c_2};\mathbf 0,\mathbf c_{i,c_4}).
\]
It remains to treat the overlap region \(C_2=[x_{i-1},x_i]\), where both basis functions are nonzero at the base point. The breakpoints in the shifted variable are
\[
\theta_{i,0}:=-\eta_\delta(x),\quad
\theta_{i,1}:=x_{i-2}-x,\quad
\theta_{i,2}:=x_{i-1}-x,\quad
\theta_{i,3}:=x_i-x,
\]
\[
\theta_{i,4}:=x_{i+1}-x,\quad
\theta_{i,5}:=\eta_\delta(x).
\]
The coefficient vectors on the corresponding subintervals are
\begin{footnotesize}
\[
\mathbf c_{i,c5}
=
\bigl(\phi_{i,1}(x)\phi_{i-1,2}(x),0,0\bigr),
\quad
\mathbf c_{i,c6}
=
\left(
-\phi_{i,1}(x)\bigl(\phi_{i-1,1}(x)-\phi_{i-1,2}(x)\bigr),
-\frac{\phi_{i,1}(x)}{h},
0
\right),
\]
\[
\mathbf c_{i,c7}
=
\left(0,0,-\frac{1}{h^2}\right),
\quad
\mathbf c_{i,c8}
=
\left(
\phi_{i-1,2}(x)\bigl(\phi_{i,1}(x)-\phi_{i,2}(x)\bigr),
\frac{\phi_{i-1,2}(x)}{h},
0
\right).
\]
\end{footnotesize}
Thus the overlap contribution is
\[
\begin{aligned}
I_{C_2}(x)
={}&
\mathcal I(\theta_{i,0},\theta_{i,1};\mathbf 0,\mathbf c_{i,c_5})
+
\mathcal I(\theta_{i,1},\theta_{i,2};\mathbf 0,\mathbf c_{i,c_6})
\\
&+
\mathcal I(\theta_{i,2},\theta_{i,3};\mathbf c_{i,c_7},\mathbf c_{i,c_7})
+
\mathcal I(\theta_{i,3},\theta_{i,4};\mathbf c_{i,c_8},\mathbf 0)
+
\mathcal I(\theta_{i,4},\theta_{i,5};\mathbf c_{i,c_5},\mathbf 0).
\end{aligned}
\]
Combining these pieces, the adjacent off-diagonal entry is
\begin{equation}
\label{eq:Aii-minus-one}
\begin{aligned}
A_{i,i-1}
={}&
\int_{B_1}w_\delta(x){I_{B}(x)}\,dx
+
\int_{B_2}w_\delta(x)I_{B}(x)\,dx
\\
&+
\int_{C_1}w_\delta(x)I_{C_1}(x)\,dx
+
\int_{C_2}w_\delta(x)I_{C_2}(x)\,dx
+
\int_{C_3}w_\delta(x)I_{C_3}(x)\,dx .
\end{aligned}
\end{equation}
If \(m_h\le1\), then \(B_1\) and \(B_2\) are empty, and only the connected and overlap contributions remain.

\paragraph{Far off-diagonal entries}
We finally consider the case \(|i-j|>1\), where the supports of \(\phi_i\) and
\(\phi_j\) are disjoint. Then
\[
\phi_i(x)\phi_j(x)=0,
\qquad
\phi_i(x+s)\phi_j(x+s)=0,
\]
and only the mixed translated-support terms contribute. Thus, when
\(x\in\operatorname{supp}\phi_i\) and \(x+s\in\operatorname{supp}\phi_j\), the integrand contains the factor \(-\phi_i(x)\phi_j(x+s)\), and the symmetric case is obtained by interchanging \(i\) and \(j\).

For a fixed \(x\), the shifted support of \(\phi_j\) is split into two intervals $(a_{j,1},b_{j,1})$ and $(a_{j,2},b_{j,2})$ with
\[
a_{j,1}:=\max\{x_{j-1}-x,-\eta_\delta(x)\},
\qquad
b_{j,1}:=\min\{x_j-x,\eta_\delta(x)\},
\]
\[
a_{j,2}:=\max\{x_j-x,-\eta_\delta(x)\},
\qquad
b_{j,2}:=\min\{x_{j+1}-x,\eta_\delta(x)\}.
\]
The corresponding coefficient vectors are
\[
\mathbf d_{j,1}^{+}
=
\left(
\phi_{j,1}(x),\,
\frac1h,\,
0
\right),
\qquad
\mathbf d_{j,1}^{-}
=
\left(
\phi_{j,1}(x),\,
-\frac1h,\,
0
\right),
\]
\[
\mathbf d_{j,2}^{+}
=
\left(
\phi_{j,2}(x),\,
-\frac1h,\,
0
\right),
\quad\;
\mathbf d_{j,2}^{-}
=
\left(
\phi_{j,2}(x),\,
\frac1h,\,
0
\right).
\]
Define
\[
I_j(x)
:=
\mathcal I(a_{j,1},b_{j,1};
\mathbf d_{j,1}^{+},\mathbf d_{j,1}^{-})
+
\mathcal I(a_{j,2},b_{j,2};
\mathbf d_{j,2}^{+},\mathbf d_{j,2}^{-}).
\]
Then, for \(|i-j|>1\),
\begin{equation}
\label{eq:Aij-remote}
A_{ij}
=
-\int_{x_{i-1}}^{x_{i+1}}
w_\delta(x)\phi_i(x)I_j(x)\,dx
-
\int_{x_{j-1}}^{x_{j+1}}
w_\delta(x)\phi_j(x)I_i(x)\,dx .
\end{equation}
\begin{remark}[Implementation and matrix structure]
\label{rem:implementation-matrix-structure}
The assembly procedure has the following consequences. Since $\mathcal{B}_\delta$ is symmetric, the stiffness matrix satisfies $A_{ij}=A_{ji}$, so only the upper triangular part needs to be assembled. The procedure is semi-analytical: the singular inner integrals in the relative variable $s=y-x$ are evaluated analytically by $\mathcal{I}$, while the outer integrals in $x$ are computed by quadrature over the active base-point regions.
Moreover, under the definition \cref{eq:mh}, the support of the $P_1$ basis functions implies
\[
A_{ij}=0
\qquad \text{for } |i-j|>m_h+2 .
\]
Thus, $A$ is a symmetric banded matrix with total bandwidth at most $2m_h+5$.
\end{remark}
\section{Consistency and error analysis}
\label{sec:analysis}


\subsection{Smooth localization and second-order consistency}
\label{subsec:smooth-horizon}

\begin{assumption}[Additional smoothness for the consistency analysis]
\label{ass:second-order-localization} 
In this subsection, we strengthen Assumption~\ref{ass:localization} by requiring
$\lambda\in C^3(\Omega)$ and
\[
|\lambda'''(x)|
\le \kappa_3 d(x)^{-2},
\qquad x\in\Omega.
\]
\end{assumption}

\begin{remark}
\label{rem:admissible-profile}
A typical admissible choice, used in the numerical experiments, is
\[
    \lambda(x)=x(1-x),
    \qquad
    q(r)=r-1+e^{-r}.
\]
It satisfies Assumption~\ref{ass:second-order-localization}.
\end{remark}

\begin{lemma}
\label{lem:eta-properties}
Let \(r_x:=\eta_\delta(x)\).
Under Assumption~\ref{ass:second-order-localization}, there exists a constant \(C>0\), independent of \(\delta\), such that for $\forall x\in \Omega$
\[
    \eta_\delta(x)\le C\delta d^2(x),
    \qquad
    |\eta_\delta'(x)|\le C\delta d(x),
    \qquad
    |\eta_\delta'(x)|^2\le C\delta \eta_\delta(x),
\]
and
\[
    |\eta_\delta''(x)|\le C\delta,
    \qquad
    \eta_\delta(x)M_3(x)\le C\delta^2,
\]
where
\[
    M_3(x)
    :=
    \sup_{\substack{y\in\Omega\\|y-x|\le2\eta_\delta(x)}}
    |\eta_\delta'''(y)|.
\]
Moreover,
\[
    \eta_\delta(x)\le\frac13q(d(x))\le\frac13d(x).
\]
\end{lemma}

\begin{proof}
On the compact range of \(\lambda\), \((A_{q,2})\) gives, 
\[
    q(\rho)\le C\rho^2,
    \qquad
    |q'(\rho)|\le C\rho,
    \qquad
    |q'(\rho)|^2\le Cq(\rho),
\]
and boundedness of \(q''\) and \(q'''\). Assumption \((A_\lambda)\) with \(k=3\) gives
\[
    \lambda(x)\le Cd(x),
    \qquad
    |\lambda'(x)|\le C,
    \qquad
    |\lambda''(x)|\le Cd^{-1}(x),
    \qquad
    |\lambda'''(x)|\le Cd^{-2}(x).
\]
Consequently,
\[
    \eta_\delta(x)=\delta q(\lambda(x))\le C\delta d^2(x).
\]
The chain rule yields:
\[
    |\eta_\delta'(x)|\le C\delta d(x),
    \qquad
    |\eta_\delta''(x)|\le C\delta,
    \qquad
    |\eta_\delta'''(x)|\le C\delta d^{-1}(x),
\]
as well as
\[
    |\eta_\delta'(x)|^2
    \le
    C\delta^2q(\lambda(x))
    =
    C\delta \eta_\delta(x).
\]
The smallness condition on \(\delta\) from Assumption~\ref{ass:localization} gives
\[
    \eta_\delta(x)\le\frac13q(d(x))\le\frac13d(x).
\]
Thus, if \(|y-x|\le2\eta_\delta(x)\), then
\[
    d(y)\ge d(x)-|y-x|\ge\frac13d(x).
\]
Applying the estimate for \(\eta_\delta'''\) at \(y\) gives
\[
    M_3(x)\le C\delta d^{-1}(x).
\]
Therefore
\[
    \eta_\delta(x)M_3(x)
    \le
    C\delta d^2(x)\cdot\delta d^{-1}(x)
    \le
    C\delta^2.
\]
\end{proof}
\begin{lemma}[Consistency of the frozen-horizon operator]
\label{lem:frozen-consistency}
Let \(u\in H^4(\Omega)\cap H_0^1(\Omega)\). For \(\varepsilon>0\), define
\begin{equation}
\label{eq:frozen-truncated-operator}
(\widetilde{\mathcal L}_\delta^\varepsilon u)(x)
:=
\mathbf 1_{\{\eta(x)>\varepsilon\}}\,
2C_\beta
\int_{\varepsilon<|t|<1}
\frac{u(x)-u(x+\eta_\delta(x)t)}
     {\eta^2_\delta(x)|t|^\beta}
\,dt .
\end{equation}
Then the strong limit
\[
\widetilde{\mathcal L}_\delta u
:=
\lim_{\varepsilon\rightarrow0}
\widetilde{\mathcal L}_\delta^\varepsilon u
\qquad\text{in }L^2(\Omega)
\]
exists and satisfies
\begin{equation}
\label{eq:frozen-consistency}
\left\|
\widetilde{\mathcal L}_\delta u-(-u'')
\right\|_{L^2(\Omega)}
\leq
C\delta^2\|u\|_{H^4(\Omega)}.
\end{equation}
\end{lemma}
\begin{proof}
By Lemma~\ref{lem:eta-properties}, \(x+\eta_\delta(x)t\in\Omega\) for
\(|t|<1\). Taylor's formula with integral remainder gives
\[
\begin{aligned}
u(x+\eta_\delta(x)t)-u(x)
={}&
\eta_\delta(x)tu'(x)
+
\frac{\eta^2_\delta(x)t^2}{2}u''(x)
+
\frac{\eta^3_\delta(x)t^3}{6}u^{(3)}(x)
\\
&+
\frac{\eta^4_\delta(x)t^4}{6}
\int_0^1(1-\theta)^3
u^{(4)}(x+\theta \eta_\delta(x)t)\,d\theta .
\end{aligned}
\]
For each fixed \(\varepsilon>0\), the truncated set
\(\{\varepsilon<|t|<1\}\) is symmetric; hence the terms containing
\(t\) and \(t^3\) vanish before the limit is taken. Consequently,
\begin{equation}
\label{eq:frozen-epsilon-decomposition}
\widetilde{\mathcal L}_\delta^\varepsilon u
=
\mathbf 1_{\{\eta>\varepsilon\}}
\left[
-(1-\varepsilon^{3-\beta})u''
+
R_{\delta,\varepsilon}
\right],
\end{equation}
where
\[
\begin{aligned}
R_{\delta,\varepsilon}(x)
:={}&
-\frac{C_\beta \eta^2_\delta(x)}{3}
\int_{\varepsilon<|t|<1}|t|^{4-\beta}
\int_0^1(1-\theta)^3
u^{(4)}(x+\theta \eta_\delta(x)t)
\,d\theta\,dt .
\end{aligned}
\]

Let \(R_\delta\) be the same expression with \(\varepsilon=0\). If \(\mathcal M\) denotes the Hardy--Littlewood maximal operator \cite{Stein1970SingularIntegrals} applied to the zero extension of \(u^{(4)}\), then
\[
|R_{\delta,\varepsilon}(x)|+|R_\delta(x)|
\leq
C\eta^2_\delta(x)\mathcal M(|u^{(4)}|)(x)
\leq
C\delta^2\mathcal M(|u^{(4)}|)(x).
\]
Since \(4-\beta>-1\), dominated convergence in \(t\), followed by the
\(L^2\)-boundedness of \(\mathcal M\), yields
\[
R_{\delta,\varepsilon}
\longrightarrow R_\delta
\qquad\text{in }L^2(\Omega).
\]
Moreover,
\(\mathbf 1_{\{\eta>\varepsilon\}}\to1\) almost everywhere and
\(1-\varepsilon^{3-\beta}\to1\). The right-hand side of
\eqref{eq:frozen-epsilon-decomposition} is dominated in \(L^2\) by
\[
|u''|+C\delta^2\mathcal M(|u^{(4)}|).
\]
Therefore,
\[
\widetilde{\mathcal L}_\delta^\varepsilon u
\longrightarrow
-u''+R_\delta
\qquad\text{in }L^2(\Omega),
\]
which proves the existence of the realized frozen operator. Finally,
\[
\left\|
\widetilde{\mathcal L}_\delta u-(-u'')
\right\|_{L^2}
=
\|R_\delta\|_{L^2}
\leq
C\delta^2\|u^{(4)}\|_{L^2}
\leq
C\delta^2\|u\|_{H^4}.
\]
\end{proof}

\begin{proposition}
\label{prop:LminusLtilde}
Under Assumption~\ref{ass:second-order-localization}, for every
\(u\in H^4(\Omega)\cap H_0^1(\Omega)\),
\begin{equation}
\label{eq:LminusLtilde-bound}
    \|(\mathcal L_\delta-\widetilde{\mathcal L}_\delta)u\|_{L^2(\Omega)}
    \le
    C\delta^2\|u\|_{H^4(\Omega)}.
\end{equation}
\end{proposition}

\begin{proof}
Fix \(x\in\Omega\), we use the notation
\[
r:=\eta_\delta(x),\qquad
\partial_x\eta_\delta=\partial_x\eta_\delta(x),\qquad
\partial_{xx}\eta_\delta=\partial_{xx}\eta_\delta(x),\qquad
\partial_{xxx}\eta_\delta=\partial_{xxx}\eta_\delta(x).
\]
to denote the first-, second-, and third-order derivatives of the localization function evaluated at the fixed point $x$. By Lemma~\ref{lem:eta-properties},
\begin{equation}
\label{eq:profile-bounds-paper}
   | r|\le C\delta,\;
    |\partial_x\eta_\delta|\le C\delta,\;
    (\partial_x\eta_\delta)^2\le C\delta \eta_\delta,\;
    |\partial_{xx}\eta_\delta|\le C\delta,\;
    r M_3\le C\delta^2.
\end{equation}
Since \(\Omega=(0,1)\), the Sobolev embedding gives
\[
H^4(\Omega)\hookrightarrow C^3(\overline\Omega),
\qquad
\|u\|_{C^k(\overline\Omega)}
\le C\|u\|_{H^4(\Omega)},
\quad k=1,2,3.
\]
For fixed $x$, we define functions
\[
    \Phi(t):=\frac{u(x)-u(x+r t)}{|t|^\beta},
    \qquad
    \vartheta(t):=\frac{\eta_\delta(x+r t)}{r},
    \qquad
    \omega(t):=
    \vartheta^{-(3-\beta)}(t),
\]
and
\[
    I_x:=
    \left\{
    t:\ x+rt,\;t\in\Omega,\
    |t|<\frac{\eta_\delta(x+r t)}{r}
    \right\}.
\]
The uniform bound on \(\|\partial_x\eta_\delta\|_\infty\) also gives \(I_x\subset(-\frac{3}{2},\frac{3}{2})\), so that all Taylor expansions below are performed on a fixed compact interval. Taylor's formula gives, for \(t\in I_x\),
\begin{equation}
\label{eq:vartheta-Taylor}
\vartheta(t)
=
1+(\partial_x\eta_\delta) t+\frac12(\partial_{xx}\eta_\delta )rt^2+e_\vartheta(t),
\qquad
|e_\vartheta(t)|
\le
CM_3\eta^2_\delta|t|^3.
\end{equation}
All constants below are independent of \(x\), \(\delta\), and \(\varepsilon\). At the truncated level, the change of variables \(y=x+r t\) gives the exact identity
\begin{equation}
\label{eq:scaled-difference-epsilon-paper}
\begin{aligned}
\bigl(
\mathcal L_\delta^\varepsilon
-
\widetilde{\mathcal L}_\delta^\varepsilon
\bigr)u(x)
=
\mathbf1_{\{\eta(x)>\varepsilon\}}
C_\beta \eta^{-2}_\delta
\Bigg[
&
\int_{\substack{t\in I_x,\,|t|>\varepsilon \vartheta(t)}}
\omega(t)\Phi(t)\,dt
\\-
&\int_{\substack{-1<t<1,\,|t|>\varepsilon}}
\Phi(t)\,dt
\Bigg].
\end{aligned}
\end{equation}

Adding and subtracting the same common-cutoff integral then gives the exact finite-\(\varepsilon\) decomposition
\begin{equation}
\label{eq:three-term-epsilon-paper}
\begin{aligned}
\bigl(
\mathcal L_\delta^\varepsilon
-
\widetilde{\mathcal L}_\delta^\varepsilon
\bigr)u(x)
=
\mathbf1_{\{\eta(x)>\varepsilon\}}
\left[
\mathcal Q_\varepsilon
+R_\omega^\varepsilon
+R_{\rm I}^\varepsilon
\right].
\end{aligned}
\end{equation}
Let
\[
    J:=(-1+\partial_x\eta_\delta,1+\partial_x\eta_\delta)
\]
and define
\begin{align}
Q_{\rm tr}^\varepsilon
:={}&
C_\beta\eta^{-2}_\delta
\left[
\int_{\substack{t\in J,\,|t|>\varepsilon}}
[1-(3-\beta)(\partial_x\eta_\delta) t]\Phi(t)\,dt
-
\int_{\substack{-1<t<1,\,|t|>\varepsilon}}
\Phi(t)\,dt
\right],
\label{eq:Q-epsilon-paper}\\
R_\omega^\varepsilon
:={}&
C_\beta \eta^{-2}_\delta
\int_{\substack{t\in I_x,\,|t|>\varepsilon}}
\bigl[\omega(t)-(1-(3-\beta)(\partial_x\eta_\delta) t)\bigr]\Phi(t)\,dt,
\label{eq:Romega-epsilon-paper}\\
R_{\rm I}^\varepsilon
:={}&
C_\beta \eta^{-2}_\delta
\left(
\int_{\substack{t\in I_x,\,|t|>\varepsilon}}
-
\int_{\substack{t\in J,\,|t|>\varepsilon}}
\right)
[1-(3-\beta)(\partial_x\eta_\delta) t]\Phi(t)\,dt.
\label{eq:RI-epsilon-paper}
\end{align}
To replace the two lower cutoffs by one common symmetric cutoff, set
\begin{equation}
\label{eq:cutoff-mismatch-paper}
E_\varepsilon(x)
:=
\int_{I_x}
\left(
\mathbf1_{\{|t|>\varepsilon\vartheta(t)\}}
-
\mathbf1_{\{|t|>\varepsilon\}}
\right)
\omega(t)\Phi(t)\,dt .
\end{equation}
Absorb the cutoff mismatch into the translated term by setting
\[
    \mathcal Q_\varepsilon
    :=
    Q_{\rm tr}^\varepsilon
    +C_\beta \eta^{-2}_\delta E_\varepsilon.
\]

Since \(\|\partial_x\eta_\delta\|_\infty<1\), \(I_x=(a,b)\), where
\[
    b=\frac{\eta_\delta(x+r b)}{r },
    \qquad
    -a=\frac{\eta_\delta(x+r a)}{r}.
\]
Taylor's formula and the mean-value theorem yield
\begin{equation}
\label{eq:endpoints-paper}
\begin{aligned}
    a&=-1+\partial_x\eta_\delta-[(\partial_x\eta_\delta)^2+\frac12(\partial_{xx}\eta_\delta )\eta_\delta]+\rho_a,\\
    b&= 1+\partial_x\eta_\delta+[(\partial_x\eta_\delta)^2+\frac12(\partial_{xx}\eta_\delta) \eta_\delta]+\rho_b.
\end{aligned}
\end{equation}
with
\begin{equation}
\label{eq:endpoint-remainder-paper}
    |\rho_a|+|\rho_b|
    \le
    C\bigl(
    |\partial_x\eta_\delta|^3+|\partial_x\eta_\delta||\partial_{xx}\eta_\delta|r
    +|\partial_{xx}\eta_\delta|^2r^2+M_3 r^2
    \bigr).
\end{equation}
We now estimate the three terms in \eqref{eq:three-term-epsilon-paper} before letting \(\varepsilon\rightarrow0\).  For $Q_{\rm tr}^\varepsilon$, apply the endpoint transport formula
\[
\begin{aligned}
\int_{-1+\partial_x\eta_\delta}^{1+\partial_x\eta_\delta}F(t)\,dt
={}&
\int_{-1}^{1}F(t)\,dt
+(\partial_x\eta_\delta)\bigl(F(1)-F(-1)\bigr)
+\frac{(\partial_x\eta_\delta)^2}{2}\bigl(F'(1)-F'(-1)\bigr)\\
&+O\!\left(
|\partial_x\eta_\delta|^3
\max_{\pm}
\sup_{|t\mp1|\le|\partial_x\eta_\delta|}
|F''(t)|
\right).
\end{aligned}
\]
to \(F=\Phi\) and \(F=t\Phi\). The common central hole cancels before this formula is applied.  Moreover, we have
\begin{equation}\label{eq:phi}
    \Phi(1)-\Phi(-1)
    =-2 r u'(x)+O(r^3)\|u\|_{C^3(\overline{\Omega})},
\end{equation}
and symmetric Taylor expansion
\begin{equation}\label{eq:tphiint}
    \int_{\varepsilon<|t|<1}t\Phi(t)\,dt
    =
    -\frac{2 r}{3-\beta}
    \bigl(1-\varepsilon^{3-\beta}\bigr)u'(x)
    +O(r^3)
\end{equation}
Applying  \eqref{eq:phi} and \eqref{eq:tphiint} into $Q_{\rm tr}^\varepsilon$, we have
\[
|  Q_{\rm tr}^\varepsilon|\le
    C\left(
    |\partial_x\eta_\delta|r+(\partial_x\eta_\delta)^2+\frac{|\partial_x\eta_\delta|^3}{r}
    \right)\|u\|_{C^3}
\]
From \eqref{eq:vartheta-Taylor} and the definition of \(\Phi\), we have 
\[
    |\vartheta(t)-1|\le C\delta|t|,
    \qquad
    |\Phi(t)|\le Cr|t|^{1-\beta}\|u\|_{C^1(\overline{\Omega})},
\]
the support of the integrand in \eqref{eq:cutoff-mismatch-paper} has measure at most \(C\delta\varepsilon^2\) and is contained in \(\{\varepsilon/2<|t|<2\varepsilon\}\), for sufficiently small \(\varepsilon\). Therefore, when $\varepsilon \rightarrow 0$,
\begin{equation}
\label{eq:cutoff-mismatch-bound-paper}
    \eta^{-2}_\delta|E_\varepsilon(x)|
    \le
    C\delta \eta^{-1}_\delta\varepsilon^{3-\beta}\|u\|_{C^1(\overline{\Omega})}
    \longrightarrow0
\end{equation}
for each fixed \(x\in\Omega\).  Combining the remaining endpoint terms with \eqref{eq:cutoff-mismatch-bound-paper} yields
\begin{equation}
\label{eq:Qepsilon-paper}
    |\mathcal Q_\varepsilon|
    \le
    C\left(
    |\partial_x\eta_\delta|r+(\partial_x\eta_\delta)^2+\frac{|\partial_x\eta_\delta|^3}{r}
    \right)\|u\|_{C^3}
    +
    C\frac{\delta}{\eta_\delta}
    \varepsilon^{3-\beta}\|u\|_{C^1(\overline{\Omega})}.
\end{equation}
For the weight term, Taylor expansion on \(|t|\le2\) gives
\[
\omega(t)-\bigl(1-(3-\beta)(\partial_x\eta_\delta) t\bigr)
=
\left(
-\frac{3-\beta}{2}(\partial_{xx}\eta_\delta) r
+\frac{(3-\beta)(4-\beta)}2(\partial_x\eta_\delta)^2
\right)t^2+\mathcal R_\omega(t),
\]
where
\[
    |\mathcal  R_\omega(t)|
    \le
    C\bigl(
    |\partial_x\eta_\delta|^3+|\partial_x\eta_\delta||\partial_{xx}\eta_\delta|r
    +|\partial_{xx}\eta_\delta|^2\eta^2_\delta+M_3\eta^2_\delta
    \bigr)|t|^3.
\]
The leading odd part cancels on the common symmetric core of \(I_x\), including the removed symmetric hole.  Indeed,
\[
\Phi(t)
=
-r u'(x)\frac{t}{|t|^\beta}
-\frac{r^2}{2}u''(x)\frac{t^2}{|t|^\beta}
+
O\!\left(r^3|t|^{3-\beta}\right)\|u\|_{C^3(\overline{\Omega})}.
\]
The first term is odd and cancels on the common symmetric core, whereas the endpoint strips have total length \(O(|\partial_x\eta_\delta|)\). Using \eqref{eq:endpoints-paper} for the remaining endpoint strips gives, for each fixed \(x\in\Omega\) and all sufficiently small \(\varepsilon\),
\[
    \left|
    \int_{\substack{t\in I_x,\,|t|>\varepsilon}}
    t^2\Phi(t)\,dt
    \right|
    \le C(r^2+|\partial_x\eta_\delta|r)\|u\|_{C^2(\overline{\Omega})}.
\]
Together with \(|\Phi(t)|\le Cr|t|^{1-\beta}\|u\|_{C^1(\overline{\Omega})}\), this yields
\begin{equation}
\label{eq:Romegaepsilon-paper}
\begin{aligned}
|R_\omega^\varepsilon|
\leq C\Big(
&(\partial_x\eta_\delta)^2
+|\partial_{xx}\eta_\delta|r
+\frac{|\partial_x\eta_\delta|^3}{r}\\
&+|\partial_x\eta_\delta||\partial_{xx}\eta_\delta|
+|\partial_{xx}\eta_\delta|^2r
+M_3r
\Big)\|u\|_{C^2(\overline{\Omega})}.
\end{aligned}
\end{equation}
Finally, set \(H(t)=(1-(3-\beta)(\partial_x\eta_\delta) t)\Phi(t)\).  For sufficiently small \(\varepsilon\), the same central hole is removed from the two integrals defining \(R_{\rm I}^\varepsilon\), and therefore cancels exactly.  The endpoint expansion \eqref{eq:endpoints-paper} and the symmetric endpoint cancellation
\[
    |H(1+\partial_x\eta_\delta)+H(-1+\partial_x\eta_\delta)|
    \le C(\eta^2_\delta+|\partial_x\eta_\delta|r)\|u\|_{C^2(\overline{\Omega})}
\]
then give
\begin{equation}
\label{eq:RIepsilon-paper}
\begin{aligned}
|R_{\rm I}^\varepsilon|
\leq
C\Big(
&(\partial_x\eta_\delta)^2
+|\partial_{xx}\eta_\delta|r
+\frac{|\partial_x\eta_\delta|^3}{r}\\
&+|\partial_x\eta_\delta||\partial_{xx}\eta_\delta|
+|\partial_{xx}\eta_\delta|^2r
+M_3r
\Big)
\|u\|_{C^2(\overline{\Omega})}.
\end{aligned}
\end{equation}
This follows by expanding \(u(x\pm \eta_\delta(1+\partial_x\eta_\delta))\) about \(x\); the two terms of order \(r\) cancel in the sum.

Combining \eqref{eq:Qepsilon-paper}--\eqref{eq:RIepsilon-paper} with \eqref{eq:profile-bounds-paper} gives, for each fixed \(x\in\Omega\) and all sufficiently small \(\varepsilon\),
\begin{equation}
\label{eq:finite-epsilon-final-paper}
\left|
\bigl(
\mathcal L_\delta^\varepsilon
-
\widetilde{\mathcal L}_\delta^\varepsilon
\bigr)u(x)
\right|
\le
C\delta^2\|u\|_{C^3(\overline{\Omega})}
+
C\frac{\delta}{r}
\varepsilon^{3-\beta}\|u\|_{C^1(\overline{\Omega})}.
\end{equation}
The last term tends to zero for fixed \(x\), since \(r>0\) and \(\beta<3\). By the strong \(L^2\)-realizations of the two operators,
\[
\bigl(
\mathcal L_\delta^\varepsilon
-
\widetilde{\mathcal L}_\delta^\varepsilon
\bigr)u
\longrightarrow
(\mathcal L_\delta-\widetilde{\mathcal L}_\delta)u
\qquad\text{in }L^2(\Omega).
\]
Hence, by the subsequence principle, there exists a sequence \(\varepsilon_j\rightarrow0\) such that
\[
\bigl(
\mathcal L_\delta^{\varepsilon_j}
-
\widetilde{\mathcal L}_\delta^{\varepsilon_j}
\bigr)u(x)
\longrightarrow
(\mathcal L_\delta-\widetilde{\mathcal L}_\delta)u(x)
\]
for almost every \(x\in\Omega\). For each such fixed \(x\), one has \(r=\eta_\delta>0\), and hence
\[
\frac{\delta}{r}\varepsilon_j^{3-\beta}\longrightarrow0.
\]
Therefore,
\[
|(\mathcal L_\delta-\widetilde{\mathcal L}_\delta)u(x)|
\le
C\delta^2\|u\|_{C^3(\overline{\Omega})}
\qquad\text{for a.e. }x\in\Omega.
\]
Integrating the almost-everywhere estimate over \(\Omega\) and using the Sobolev embedding
\(H^4(\Omega)\hookrightarrow C^3(\overline{\Omega})\), we obtain
\[
\|(\mathcal L_\delta-\widetilde{\mathcal L}_\delta)u\|_{L^2(\Omega)}
\le
C\delta^2\|u\|_{H^4(\Omega)}.
\]

\end{proof}

\begin{theorem}[Second-order operator consistency]
\label{thm:operator-consistency}
Under Assumption~\ref{ass:second-order-localization}, for every
\(u\in H^4(\Omega)\cap H_0^1(\Omega)\),
\begin{equation}
\label{eq:operator-consistency}
    \|\mathcal L_\delta u-(-u'')\|_{L^2(\Omega)}
    \le
    C\delta^2\|u\|_{H^4(\Omega)},
\end{equation}
where \(C\) is independent of \(u\) and \(\delta\).
\end{theorem}

\begin{proof}
The decomposition
\[
    \mathcal L_\delta u-(-u'')
    =
    (\mathcal L_\delta-\widetilde{\mathcal L}_\delta)u
    +
    \bigl(\widetilde{\mathcal L}_\delta u-(-u'')\bigr)
\]
and the triangle inequality reduce the result to Proposition~\ref{prop:LminusLtilde} and Lemma~\ref{lem:frozen-consistency}.
\end{proof}

\subsection{Variational consistency and finite element error estimates}
\label{subsec:error-estimates}


\begin{proposition}[Uniform continuity and coercivity]
\label{prop:Bdelta-continuity-coercivity}
There exist \(c_0,C_0>0\), independent of \(\delta\), such that
\[
    |\mathcal B_\delta(u,v)|
    \le
    C_0
    \|u\|_{\mathfrak W^{\beta,2}[\delta;q](\Omega)}
    \|v\|_{\mathfrak W^{\beta,2}[\delta;q](\Omega)}
\]
for all \(u,v\in\mathfrak W_0^{\beta,2}[\delta;q](\Omega)\), and
\[
    \mathcal B_\delta(v,v)
    \ge
    c_0\|v\|_{\mathfrak W^{\beta,2}[\delta;q](\Omega)}^2
\]
for all \(v\in\mathfrak W_0^{\beta,2}[\delta;q](\Omega)\).
\end{proposition}

\begin{proof}
The first estimate follows from Cauchy--Schwarz in the nonlocal energy
measure. The second follows from the uniform nonlocal Poincar\'e
inequality recalled in Section~\ref{subsec:variational-formulation}.
\end{proof}

\begin{proposition}[Variational consistency]
\label{prop:variational-consistency}
Let \(u\in H^4(\Omega)\cap H_0^1(\Omega)\). For
\(v\in\mathfrak W_0^{\beta,2}[\delta;q](\Omega)\), define
\[
    \mathcal B_0(u,v):=(-u'',v).
\]
Then
\begin{equation}
\label{eq:variational-consistency-L2}
    |\mathcal B_\delta(u,v)-\mathcal B_0(u,v)|
    \le
    C\delta^2
    \|u\|_{H^4(\Omega)}
    \|v\|_{L^2(\Omega)},
\end{equation}
and consequently
\begin{equation}
\label{eq:variational-consistency-Tdelta}
    |\mathcal B_\delta(u,v)-\mathcal B_0(u,v)|
    \le
    C\delta^2
    \|u\|_{H^4(\Omega)}
    \|v\|_{\mathfrak W^{\beta,2}[\delta;q](\Omega)}.
\end{equation}
\end{proposition}

\begin{proof}
The nonlocal Green identity
\eqref{eq:nonlocal-green} gives
\[
    \mathcal B_\delta(u,v)=(\mathcal L_\delta u,v).
\]
Therefore
\[
    \mathcal B_\delta(u,v)-\mathcal B_0(u,v)
    =
    (\mathcal L_\delta u-(-u''),v).
\]
Cauchy--Schwarz and Theorem~\ref{thm:operator-consistency} prove
\eqref{eq:variational-consistency-L2}. The second estimate follows from
\(\|v\|_{L^2}\le\|v\|_{\mathfrak W^{\beta,2}[\delta;q]}\).
\end{proof}

The comparison of heterogeneous seminorms
\cite[Theorem~2.2]{scott2024nonlocal}, the classical-to-nonlocal
embedding \cite[Corollary~5.3]{scott2024nonlocal}, and the classical
Poincar\'e inequality imply
\begin{equation}
\label{eq:uniform-embedding}
    \|v\|_{\mathfrak W^{\beta,2}[\delta;q](\Omega)}
    \le
    C_{\rm emb}\|v\|_{H_0^1(\Omega)}
    \qquad
    \forall v\in H_0^1(\Omega),
\end{equation}
where \(C_{\rm emb}\) is independent of \(\delta\).

\begin{theorem}[Strang-type error estimate]
\label{thm:strang}
Let \(u_0\in H^2(\Omega)\cap H_0^1(\Omega)\) solve the local weak problem, and let
\(u_{\delta,h}\in V_h\subset\mathfrak W_0^{\beta,2}[\delta;q](\Omega)\) solve
\eqref{eq:discrete-weak-problem}. Then
\begin{equation}
\label{eq:strang}
\begin{aligned}
\|u_{\delta,h}-u_0\|_{\mathfrak W^{\beta,2}[\delta;q](\Omega)}
\le C\Bigg(
&\inf_{v_h\in V_h}
\|u_0-v_h\|_{\mathfrak W^{\beta,2}[\delta;q](\Omega)}\\
&+
\sup_{0\ne w_h\in V_h}
\frac{
|\mathcal B_\delta(u_0,w_h)-\mathcal B_0(u_0,w_h)|
}{
\|w_h\|_{\mathfrak W^{\beta,2}[\delta;q](\Omega)}
}
\Bigg),
\end{aligned}
\end{equation}
where \(C\) is independent of \(h\) and \(\delta\).
\end{theorem}
\begin{proof}
Since
\[
    u_0\in H^2(\Omega)\cap H_0^1(\Omega)
    \qquad\text{and}\qquad
    -u_0''=f
    \quad\text{in }L^2(\Omega),
\]
we extend the first argument of the local form by setting
\[
    \mathcal B_0(u_0,w)
    :=
    (-u_0'',w)_{L^2(\Omega)}
    =
    (f,w)_{L^2(\Omega)},
    \qquad
    w\in\mathfrak W_0^{\beta,2}[\delta;q](\Omega).
\]
This agrees with the usual local bilinear form whenever
\(w\in H_0^1(\Omega)\).

Let \(v_h\in V_h\) be arbitrary and set \(
    z_h:=u_{\delta,h}-v_h\in V_h.
\)
Using the discrete problem and the identity
\(\mathcal B_0(u_0,z_h)=(f,z_h)\), we obtain
\[
\begin{aligned}
    \mathcal B_\delta(z_h,z_h)
    &=
    \mathcal B_\delta(u_{\delta,h},z_h)
    -
    \mathcal B_\delta(v_h,z_h)       =
    (f,z_h)
    -
    \mathcal B_\delta(v_h,z_h)                                      \\
    &=
    \mathcal B_0(u_0,z_h)
    -
    \mathcal B_\delta(v_h,z_h)                                =
    \mathcal B_0(u_0,z_h)
    -
    \mathcal B_\delta(u_0,z_h)
    +
    \mathcal B_\delta(u_0-v_h,z_h).
\end{aligned}
\]

Define the consistency functional
\[
    R_h(w_h)
    :=
    \mathcal B_\delta(u_0,w_h)
    -
    \mathcal B_0(u_0,w_h),
    \qquad
    w_h\in V_h.
\]
It follows that
\[
    \mathcal B_\delta(z_h,z_h)
    =
    -R_h(z_h)
    +
    \mathcal B_\delta(u_0-v_h,z_h).
\]
By the uniform coercivity and continuity of
\(\mathcal B_\delta\),
\[
\begin{aligned}
    c_0\|z_h\|_{\mathfrak W^{\beta,2}[\delta;q](\Omega)}^2
    &\le
    \mathcal B_\delta(z_h,z_h)     \le
    |R_h(z_h)|
    +
    C_0
    \|u_0-v_h\|_{\mathfrak W^{\beta,2}[\delta;q](\Omega)}
    \|z_h\|_{\mathfrak W^{\beta,2}[\delta;q](\Omega)}.
\end{aligned}
\]
Moreover,
\[
    |R_h(z_h)|
    \le
    \left(
        \sup_{0\ne w_h\in V_h}
        \frac{
            \left|
                \mathcal B_\delta(u_0,w_h)
                -
                \mathcal B_0(u_0,w_h)
            \right|
        }{
            \|w_h\|_{\mathfrak W^{\beta,2}[\delta;q](\Omega)}
        }
    \right)
    \|z_h\|_{\mathfrak W^{\beta,2}[\delta;q](\Omega)}.
\]
If \(z_h\ne0\), division by
\(\|z_h\|_{\mathfrak W^{\beta,2}[\delta;q](\Omega)}\) therefore gives
\begin{equation}
\label{eq:strang-zh-bound}
\begin{aligned}
    \|z_h\|_{\mathfrak W^{\beta,2}[\delta;q](\Omega)}
    \le C\Bigg(
        \|u_0-v_h\|_{\mathfrak W^{\beta,2}[\delta;q](\Omega)}
        +
        \sup_{0\ne w_h\in V_h}
        \frac{
            \left|
                \mathcal B_\delta(u_0,w_h)
                -
                \mathcal B_0(u_0,w_h)
            \right|
        }{
            \|w_h\|_{\mathfrak W^{\beta,2}[\delta;q](\Omega)}
        }
    \Bigg),
\end{aligned}
\end{equation}
where \(C>0\) depends only on the uniform continuity and
coercivity constants \(C_0\) and \(c_0\). The case \(z_h=0\)
is immediate. Finally, the triangle inequality yields
\begin{equation}
\label{eq:strang-triangle}
\begin{aligned}
    \|u_{\delta,h}-u_0\|_{\mathfrak W^{\beta,2}[\delta;q](\Omega)}
    &\le
    \|u_{\delta,h}-v_h\|_{\mathfrak W^{\beta,2}[\delta;q](\Omega)}
    +
    \|u_0-v_h\|_{\mathfrak W^{\beta,2}[\delta;q](\Omega)}                          \\
    &=
    \|z_h\|_{\mathfrak W^{\beta,2}[\delta;q](\Omega)}
    +
    \|u_0-v_h\|_{\mathfrak W^{\beta,2}[\delta;q](\Omega)}.
\end{aligned}
\end{equation}
Combining \eqref{eq:strang-zh-bound} and
\eqref{eq:strang-triangle}, and then taking the infimum over
\(v_h\in V_h\), proves
\[
\begin{aligned}
    \|u_{\delta,h}-u_0\|_{\mathfrak W^{\beta,2}[\delta;q](\Omega)}
    \le C\Bigg(
        \inf_{v_h\in V_h}
        \|u_0-v_h\|_{\mathfrak W^{\beta,2}[\delta;q](\Omega)}
        \\+
        \sup_{0\ne w_h\in V_h}
        \frac{
            \left|
                \mathcal B_\delta(u_0,w_h)
                -
                \mathcal B_0(u_0,w_h)
            \right|
        }{
            \|w_h\|_{\mathfrak W^{\beta,2}[\delta;q](\Omega)}
        }
    \Bigg).
\end{aligned}
\]
\end{proof}

\begin{corollary}[Energy and \(L^2\) error estimates]
\label{cor:energy-total}
Let \(u_0\in H^4(\Omega)\cap H_0^1(\Omega)\) be the local solution. Then
\begin{equation}
\label{eq:energy-total}
    \|u_{\delta,h}-u_0\|_{\mathfrak W^{\beta,2}[\delta;q](\Omega)}
    \le
    C(h+\delta^2)\|u_0\|_{H^4(\Omega)},
\end{equation}
and
\begin{equation}
\label{eq:L2-total}
    \|u_{\delta,h}-u_0\|_{L^2(\Omega)}
    \le
    C(h+\delta^2)\|u_0\|_{H^4(\Omega)}.
\end{equation}
\end{corollary}

\begin{proof}
Let \(I_hu_0\in V_h\) be the nodal interpolant. By
\eqref{eq:uniform-embedding} and the standard piecewise-linear
interpolation estimate,
\[
\begin{aligned}
\inf_{v_h\in V_h}
\|u_0-v_h\|_{\mathfrak W^{\beta,2}[\delta;q]}
&\le
\|u_0-I_hu_0\|_{\mathfrak W^{\beta,2}[\delta;q]}\le
C\|u_0-I_hu_0\|_{H_0^1}
\le
Ch\|u_0\|_{H^2}.
\end{aligned}
\]
The residual term in Theorem~\ref{thm:strang} is bounded by
Proposition~\ref{prop:variational-consistency}. This proves
\eqref{eq:energy-total}. Estimate \eqref{eq:L2-total} follows from
\(\|v\|_{L^2}\le\|v\|_{\mathfrak W^{\beta,2}[\delta;q]}\).
\end{proof}
\begin{remark}[Classical error norms]
\label{rem:improved-L2-and-classical-norms}
The estimate in Corollary~\ref{cor:energy-total} is not expected to be sharp. Indeed, the numerical experiments below exhibit higher convergence rates in the classical \(H^1\)- and \(L^2\)-norms than those predicted by the present analysis. In particular, an \(L^2\)-estimate of order \(O(h^2+\delta^2)\) would require additional regularity theory for the associated nonlocal dual problem, which is beyond the scope of this paper. These improved rates should therefore be regarded as numerical observations rather than consequences of Corollary~\ref{cor:energy-total}.
\end{remark}
\begin{remark}{Asymptotic compatibility}
For sufficiently regular local solutions, estimate \eqref{eq:L2-total} immediately implies that, for every pair of
sequences \(\delta_n\rightarrow0\) and \(h_n\rightarrow0\),
\[
    \|u_{\delta_n,h_n}-u_0\|_{L^2(0,1)}
    \longrightarrow0.
\]
Thus the method recovers the local solution along every path
\((\delta,h)\to(0,0)\) for
\(u_0\in H^4(0,1)\cap H_0^1(0,1)\).
\end{remark}

\section{Numerical experiments}
\label{sec:experiments}
We present numerical experiments to validate the preceding analysis and illustrate the behavior of the numerical solutions in the local-limit regime.  The reference solution \(u_0\) solves the local Dirichlet problem
\begin{equation}\label{experiment1}
\begin{cases}
    - u''(x) = f(x)& \text{in } (0, 1), \\
    u(0) = u(1) = 0.
\end{cases}
\end{equation}
The experiments are organized into two examples. Example~1 investigates the boundary behavior and convergence under the coupled scalings \(\delta=h\) and \(\delta=\sqrt{h}\), alongside additional tests with fixed nonlocal horizon $\delta$ and singularity parameter $\beta$. Example~2 provides an additional boundary-sensitive test in which the exact solution has nonzero slopes at both endpoints.

As described in Section~3, singular inner moments in the stiffness matrix are evaluated analytically, whereas remaining outer integrals are computed using a 4-point Gauss--Legendre rule. Standard $L^2$- and $H^1$-error norms are evaluated elementwise with a 3-point Gauss--Legendre rule. For the nonlocal energy norm , outer and inner integrals in $\mathcal B_\delta(e_{\delta,h},e_{\delta,h})$ are computed using 4- and 4-point Gauss--Legendre rules, respectively, after splitting at $y=x$.

Convergence rates are computed by $\operatorname{order} = \log_2(E_{h_k}/E_{h_{k+1}})$ under uniform mesh refinement $h_{k+1}=h_k/2$. Elementwise derivatives $u_{\delta,h}'$ of the piecewise-linear solution are assigned right-element values at interior nodes for visualization only, without affecting integrated error norms.

\subsection{Example 1}
The first experiment adopts the same benchmark problem as that considered in~\cite{tian2013analysis}.
The corresponding local exact solution and source term are \(u_0(x)=x^2(1-x^2),\; f(x)=12x^2-2.\) For the theoretical convergence tests, we use the smooth localization profile
introduced in Section~\ref{sec:analysis} as 
$q(r)=r-1+e^{-r}$ and $\lambda(x)=x(1-x)$, so that
\begin{equation}
\eta_\delta(x)=\delta q(\lambda(x))=\delta(x(1-x)-1+e^{-x(1-x)}).
\end{equation} 
We first compare the near-boundary behavior of  the heterogeneous localization model and the classical nonlocal model equipped with the homogeneous exterior volume constraint. Both models are discretized by continuous piecewise linear finite  element method. Figure~\ref{fig:ex1_comparison} displays the numerical derivatives and pointwise errors for $h=2^{-8}$ and $\delta=h$. The classical nonlocal formulation exhibits visible boundary effects in the numerical derivative. In contrast, the heterogeneous localization model produces a smoother boundary profile and smaller near-boundary errors. This comparison provides an initial illustration of the boundary-regularizing effect of heterogeneous localization.
\begin{figure}[!htbp]
    \centering
    \begin{minipage}[t]{0.45\textwidth}
        \centering
        \includegraphics[width=\textwidth]{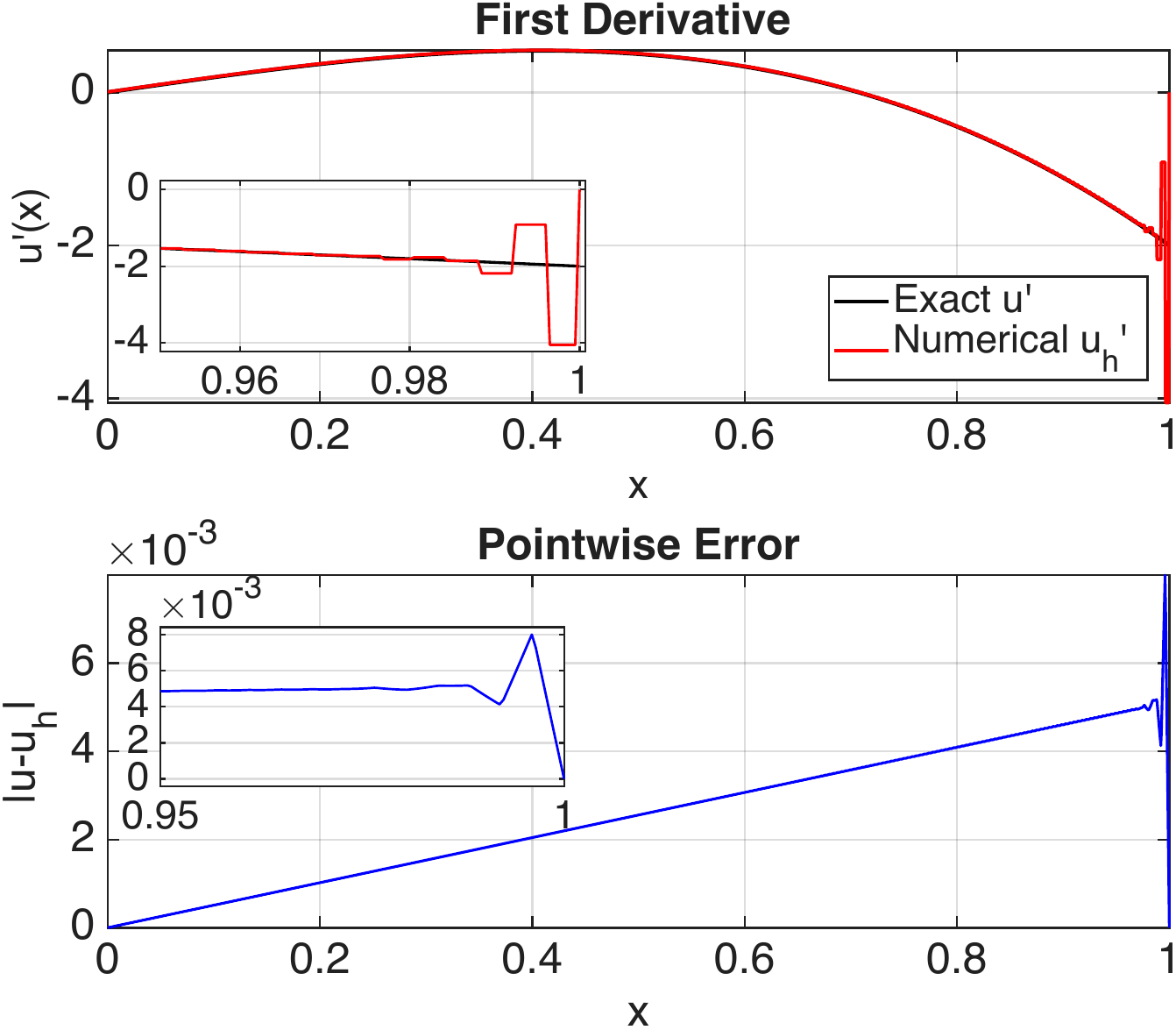} 
    \end{minipage}
    \hfill
    \begin{minipage}[t]{0.46\textwidth}
        \centering
        \includegraphics[width=\textwidth]{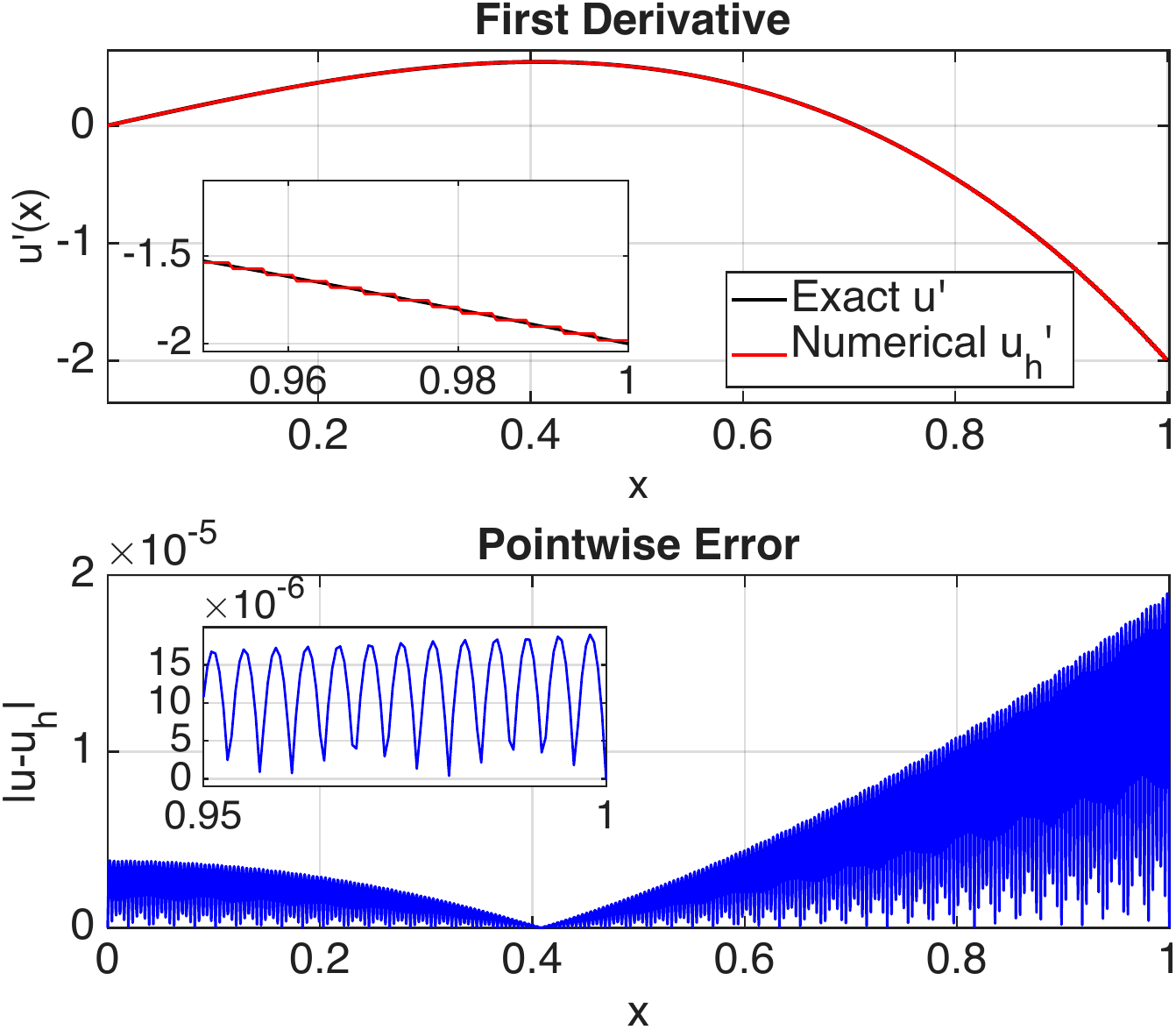} 
    \end{minipage}
    \caption{Comparison of numerical solution derivatives for traditional nonlocal model (left) and the proposed model (right) in Example 1.}
    \label{fig:ex1_comparison}
\end{figure}

Next, we test the convergence behavior under the coupled scalings $\delta=h$ and $\delta=\sqrt{h}$. Tables~\ref{tab:ex1_errors_delta_h} and~\ref{tab:ex1_errors_delta_h2} show the errors and convergence orders for $\delta=h$ and $\delta=\sqrt{h}$. The nonlocal energy errors agree with the first-order estimate in Section~\ref{sec:analysis} with the nonlocal energy norm defined in \eqref{eq:nonlocal_energynorm}. We also include the classical \(L^2\)- and \(H^1\)-errors as reference quantities for comparison with standard finite element behavior. To explore the influence of the singularity parameter $\beta$, we supplement results with $\beta=2.5$ and $\delta=h$ in Table \ref{tab:ex1_errors_delta_h_beta25}. To ensure that the interaction neighborhood spans more than one mesh cell, we additionally consider the scaling $\delta=40h$. For the localization profile used here, \(m_h=2.\) The observed convergence rates in  Table \ref{tab:ex1_errors_delta_2} remain consistent with the theoretical estimate, indicating that the behavior is not restricted to the single-cell interaction regime. Finally, Table \ref{tab:ex1_errors_delta_fixed} reports convergence data when the nonlocal horizon is fixed at $\delta=\frac{1}{2^5}$ independent of mesh size $h$.
\begin{table}[htbp]
\centering
\caption{\emph{Errors and convergence orders with $\delta=h$ and $\beta=1$.}}
\label{tab:ex1_errors_delta_h}
\setlength{\tabcolsep}{5.7pt}
\begin{tabular}{ccccccc}
\hline
$h$ & $L^2$-error & Order
& $H^1$-error & Order
& \shortstack{Nonlocal error}& Order \\
\hline
$2^{-5}$& $3.65\times10^{-4}$ & --
& $3.73\times10^{-2}$ & --
& $3.39\times10^{-2}$ & -- \\

$2^{-6}$& $9.13\times10^{-5}$ & 2.00
& $1.86\times10^{-2}$ & 1.01
& $1.76\times10^{-2}$ & 0.95 \\

$2^{-7}$& $2.28\times10^{-5}$ & 2.00
& $9.27\times10^{-3}$ & 1.00
& $7.49\times10^{-3}$ & 1.23 \\

$2^{-8}$& $5.71\times10^{-6}$ & 2.00
& $4.63\times10^{-3}$ & 1.00
& $3.83\times10^{-3}$ & 0.97 \\

$2^{-9}$& $1.43\times10^{-6}$ & 2.00
& $2.31\times10^{-3}$ & 1.00
& $2.03\times10^{-3}$ & 0.92 \\

$2^{-10}$& $3.57\times10^{-7}$ & 2.00
& $1.16\times10^{-3}$ & 1.00
& $1.00\times10^{-3}$ & 1.02 \\
\hline
\end{tabular}
\end{table}
\begin{table}[htbp]
\centering
\caption{\emph{Errors and convergence orders with $\delta = \sqrt{h}$ and $\beta = 1$.}}
\label{tab:ex1_errors_delta_h2}
\setlength{\tabcolsep}{5.7pt}
\begin{tabular}{ccccccc}
\hline
$h$ & $L^2$-error & Order & $H^1$-error & Order & Nonlocal error & Order \\
\hline
$2^{-5}$& $3.65\times10^{-4}$ & --   & $3.37 \times 10^{-2}$& --   & $2.41 \times 10^{-1}$& --   \\
$2^{-6}$& $9.13\times10^{-5}$ & 2.00& $1.86 \times 10^{-2}$& 1.01& $1.12 \times 10^{-1}$& 1.10\\
$2^{-7}$& $2.23\times10^{-5}$& 2.03& $9.26 \times 10^{-3}$& 1.00& $4.13 \times 10^{-2}$& 1.44\\
$2^{-8}$& $5.36\times10^{-6}$& 2.05& $4.63 \times 10^{-3}$& 1.00& $3.77 \times 10^{-2}$ & 1.01\\
$2^{-9}$& $1.31\times10^{-6}$& 2.03& $2.31 \times 10^{-3}$& 1.00& $1.02 \times 10^{-2}$& 0.99\\
$2^{-10}$& $3.29\times10^{-7}$& 1.99& $1.11 \times 10^{-3}$& 1.00 & $4.32 \times 10^{-3}$& 1.24\\
\hline
\end{tabular}
\end{table}
\begin{table}[htbp]
\centering
\caption{\emph{Errors and convergence orders with $\delta=h$ and $\beta=2.5$.}}
\label{tab:ex1_errors_delta_h_beta25}
\setlength{\tabcolsep}{5.7pt}
\begin{tabular}{ccccccc}
\hline
$h$ & $L^2$-error & Order
& $H^1$-error & Order
& \shortstack{Nonlocal error}& Order \\
\hline
$2^{-5}$& $3.65\times10^{-4}$ & --
& $3.73\times10^{-2}$ & --
& $3.42\times10^{-2}$ & -- \\

$2^{-6}$& $9.13\times10^{-5}$ & 2.00
& $1.86\times10^{-2}$ & 1.01
& $1.76\times10^{-2}$ & 0.96 \\

$2^{-7}$& $2.28\times10^{-5}$ & 2.00
& $9.27\times10^{-3}$ & 1.00
& $7.51\times10^{-3}$ & 1.23 \\

$2^{-8}$& $5.71\times10^{-6}$ & 2.00
& $4.63\times10^{-3}$ & 1.00
& $3.70\times10^{-3}$ & 1.02 \\

$2^{-9}$& $1.43\times10^{-6}$ & 2.00
& $2.31\times10^{-3}$ & 1.00
& $1.94\times10^{-3}$ & 0.93 \\

$2^{-10}$& $3.57\times10^{-7}$ & 2.00
& $1.16\times10^{-3}$ & 1.00
& $9.66\times10^{-4}$ & 1.01 \\
\hline
\end{tabular}
\end{table}
\begin{table}[htbp]
\centering
\caption{\emph{Errors and convergence orders with $\delta=40h$ and $\beta=1$.}}
\label{tab:ex1_errors_delta_2}
\setlength{\tabcolsep}{5.7pt}
\begin{tabular}{ccccccc}
\hline
$h$ & $L^2$-error & Order
& $H^1$-error & Order
& \shortstack{Nonlocal error}& Order \\
\hline
$2^{-5}$& $9.10\times10^{-4}$& --
& $3.73\times10^{-2}$ & --
& $1.25\times10^{-2}$& -- \\

$2^{-6}$& $2.28\times10^{-4}$& 2.00
& $1.86\times10^{-2}$ & 1.01
& $5.74\times10^{-3}$& 1.12\\

$2^{-7}$& $5.71\times10^{-5}$& 1.99& $9.27\times10^{-3}$ & 1.00
& $2.62\times10^{-3}$& 1.13\\

$2^{-8}$& $1.43\times10^{-5}$& 2.00
& $4.63\times10^{-3}$ & 1.00
& $1.38\times10^{-3}$& 0.92\\

$2^{-9}$& $3.57\times10^{-6}$& 2.00& $2.31\times10^{-3}$ & 1.00
& $7.38\times10^{-4}$& 0.92 \\

$2^{-10}$& $8.92\times10^{-7}$& 2.00& $1.16\times10^{-3}$ & 0.99& $3.95\times10^{-4}$& 0.90\\
\hline
\end{tabular}
\end{table}
\begin{table}[htbp]
\centering
\caption{\emph{Errors and convergence orders with $\delta = \frac{1}{2^5}$ and $\beta = 1$.}}
\label{tab:ex1_errors_delta_fixed}
\setlength{\tabcolsep}{5.7pt}
\begin{tabular}{ccccccc}
\hline
$h$ & $L^2$-error & Order & $H^1$-error & Order & Nonlocal error & Order \\
\hline
$2^{-5}$& $3.65 \times 10^{-4}$ & -- & $3.73 \times 10^{-2}$ & -- & $3.40 \times 10^{-2}$ & -- \\
$2^{-6}$& $9.13 \times 10^{-5}$ & 2.00 & $1.86 \times 10^{-2}$ & 1.01 & $1.59 \times 10^{-2}$ & 1.09 \\
$2^{-7}$& $2.28 \times 10^{-5}$ & 2.00 & $9.27 \times 10^{-3}$ & 1.00 & $5.21 \times 10^{-3}$ & 1.61 \\
$2^{-8}$& $5.71 \times 10^{-6}$ & 2.00 & $4.63 \times 10^{-3}$ & 1.00 & $2.44 \times 10^{-3}$ & 1.09 \\
$2^{-9}$& $1.34 \times 10^{-6}$ & 2.09 & $2.31 \times 10^{-3}$ & 1.00 & $1.14 \times 10^{-3}$ & 1.10 \\
$2^{-10}$& $3.29 \times 10^{-7}$ & 2.03 & $1.16 \times 10^{-3}$ & 1.00 & $4.32 \times 10^{-4}$ & 1.40 \\
\hline
\end{tabular}
\end{table}

\subsection{Example 2}
\label{subsec:u_exact2}
The second experiment uses the local benchmark solution $u_0(x)=e^{(x (1 - x))}  \sin(\pi x)$, with \(f(x)=e^{x (1 - x) } ((\pi^2 + 2 - (1 - 2x)^2)\sin(\pi x) - 2\pi(1 - 2x)\cos(\pi x))\). We compare the numerical derivatives and pointwise errors of the classical nonlocal model and the heterogeneous localization model. The fixed-horizon solution exhibits boundary-localized discrepancies at both endpoints, whereas the localized model yields a smoother derivative profile and reduced near-boundary errors.
\begin{figure}[!htbp]
    \centering
    \begin{minipage}[t]{0.45\textwidth}
        \centering
        \includegraphics[
            width=1\textwidth,
            trim=0 5.5cm 62 6cm,
            clip
        ]{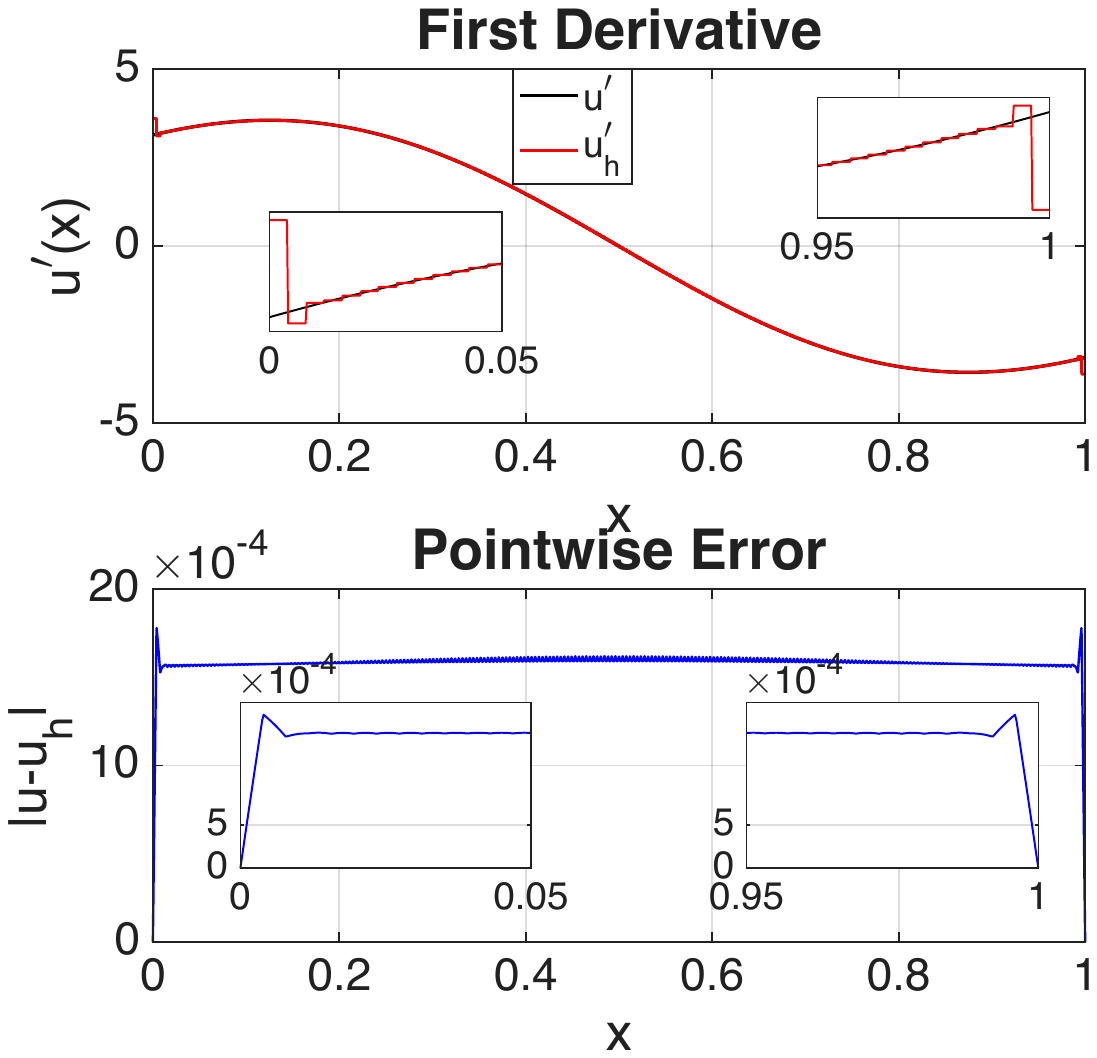}
    \end{minipage}
    \hfill
    \begin{minipage}[t]{0.46\textwidth}
        \centering
        \includegraphics[
            width=1\textwidth,
            trim=1.5 5.5cm 62 6cm,
            clip
        ]{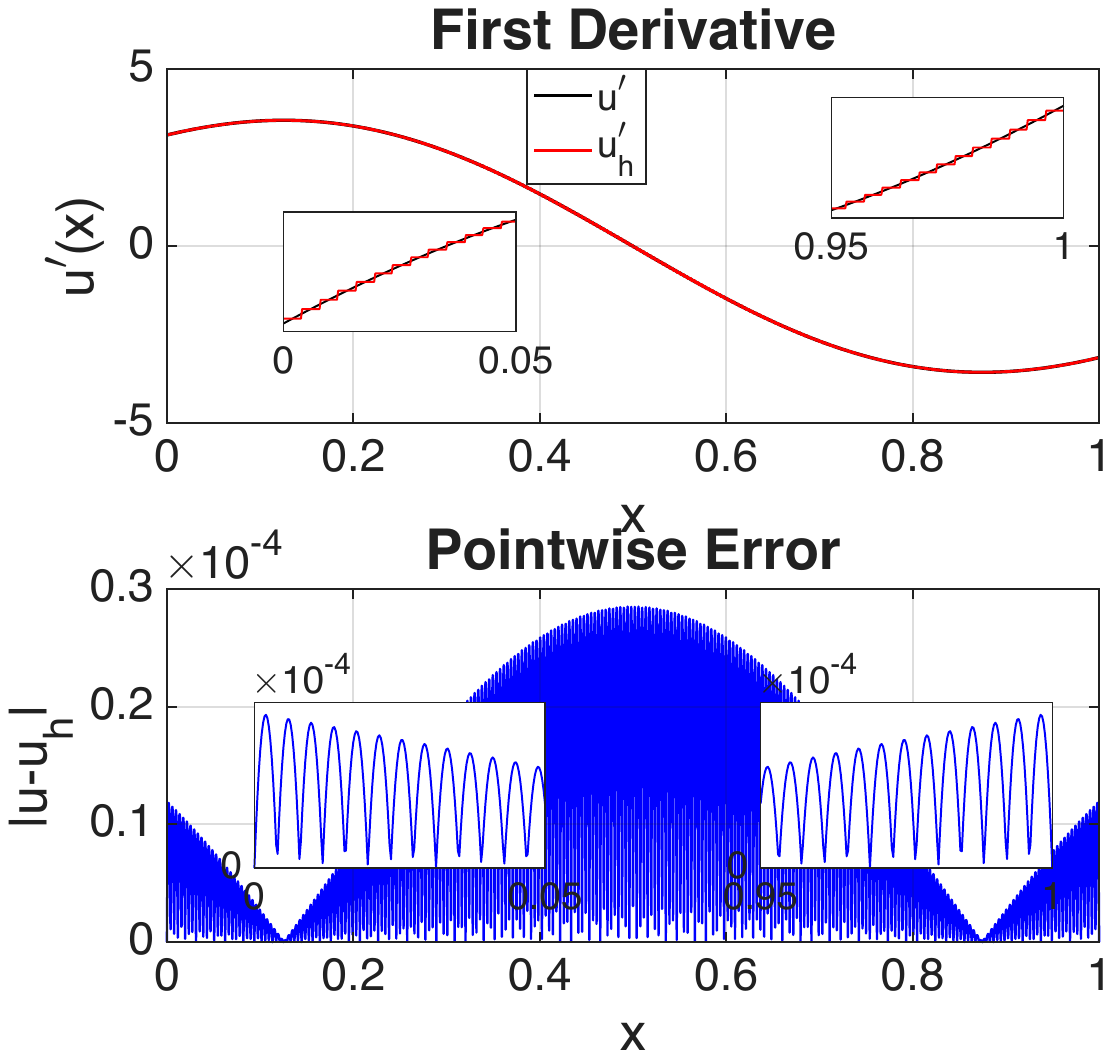}
    \end{minipage}
    \caption{Comparison of numerical solution derivatives for the traditional
    nonlocal model (left) and the proposed heterogeneous localization model
    (right) in Example~2.}
    \label{fig:ex2_comparison}
\end{figure}

\section{Conclusions}
\label{sec:conclusions}
This paper developed and analyzed a conforming finite element method for a one-dimensional nonlocal Poisson problem with heterogeneous localization and homogeneous local Dirichlet boundary conditions. The method is based on a geometric decomposition of spatially varying interaction regions, which enables a semi-analytical assembly of the singular-kernel stiffness matrix. Under suitable smoothness assumptions on the localization profile, we established second-order consistency of the realized nonlocal operator with the local Dirichlet Laplacian, derived the corresponding variational consistency estimate, and obtained the finite element error bound. Numerical experiments confirm the expected convergence behavior and show reduced near-boundary discrepancies compared with the volume-constraint formulation.

Future work includes extensions to different local boundary conditions, as well as to higher-dimensional geometries, where
the interaction-region geometry and matrix assembly become more challenging.
The proposed geometric assembly strategy may also serve as a building block for
more general nonlocal and peridynamic discretizations with spatially varying
interaction ranges.

\bibliographystyle{siamplain}
\bibliography{references}
\end{document}